\documentclass[12pt,reqno]{amsart}
\usepackage{amsmath,amssymb}
\usepackage[english]{babel}
\usepackage{hyperref}
\usepackage{cleveref}
\usepackage{cancel}
\usepackage{caption}
\usepackage{amssymb,amsmath,euscript,enumerate,tikz}
\usepackage[margin=1in]{geometry}
\usepackage{graphicx}
\usepackage{float}
\usepackage{subcaption}

\usepackage{pgf,tikz}

\usepackage{mathrsfs}
\usepackage{upgreek}
\usepackage[normalem]{ulem}

\def\NZQ{\Bbb}               
\def\NN{{\NZQ N}}

\def\RR{{\mathcal R}}
\def\KK{{\NZQ K}}
\def\FF{{\mathcal F}}
\def\AA{{\mathcal A}}
\def\MP{{\NZQ P}}

\def\MC{{\mathcal C}}

\def\MP{{\mathcal P}}

\def\MQ{{\mathcal Q}}

\def\MS{{\mathcal S}}

\def\MF{{\mathcal F}}

\def\MS{{\mathcal S}}

\def\FF{{\mathcal F}}

\def\MM{{\mathcal M}}

\newcommand{\PP}{\mathcal{P}}
\newcommand \rk{\operatorname{rk}}
\newcommand \reg{\operatorname{reg}}

\newcommand \soc{\operatorname{Soc}}

\newcommand \pd{\operatorname{pdim}}

\newcommand \ini{\operatorname{in}}

\newcommand{\rad}{1.5 pt}

\newtheorem{theorem}{Theorem}[section]
\newtheorem{definition}[theorem]{Definition}
\newtheorem{lemma}[theorem]{Lemma}
\newtheorem{proposition}[theorem]{Proposition}
\newtheorem{example}[theorem]{Example}

\newtheorem{question}[theorem]{Question}

\newtheorem{corollary}[theorem]{Corollary}
\newtheorem{conjecture}[theorem]{Conjecture}

\newtheorem{notation}[theorem]{Notation}

\newtheorem{remark}[theorem]{Remark}

\newcommand{\y}{y}

\begin{document}
\title[Level and pseudo-Gorenstein path polyominoes]{Level and pseudo-Gorenstein path polyominoes}

	\author[Giancarlo Rinaldo]{Giancarlo Rinaldo}
\email{giancarlo.rinaldo@unime.it}
\address{Department	of Mathematics, Informatics, Physics and Earth Science, University of Messina, Viale F. Stagno d’Alcontres, 31, Messina, 98166, Italy}
	\author[Francesco Romeo]{Francesco Romeo}
\email{francesco.romeo-3@unitn.it}
\address{Department of Mathematics, University of Trento, via Sommarive, 14, 38123 Povo (Trento), Italy}
\author[Rajib Sarkar]{Rajib Sarkar}
\email{rajib.sarkar63@gmail.com}
\address{Stat-Math Unit, Indian Statistical Institute, 203 B.T. Road, Kolkata--700108, India.}

\begin{abstract}
We classify path polyominoes which are level and pseudo-Gorenstein. Moreover, we compute all level and pseudo-Gorenstein simple thin polyominoes with rank less than or equal to 10. We also compute the regularity of the pseudo-Gorenstein simple thin polyominoes in relation to their rank.
\end{abstract}
\keywords{ Simple thin polyomino, path polyomino, Cohen-Macaulay, Level, pseudo-Gorenstein}
\thanks{AMS Subject Classification (2020): 13D02, 13C14, 05E40}
\maketitle

\section{Introduction}

The ideals generated by a subset of $2$-minors of an $m\times n$ matrix of indeterminates are an intensively-studied class of binomial ideals, due to their applications in algebraic statistics. Among these ideals, one finds the determinantal ideals, see, for
instance, \cite{BV} and its references to original articles, the ladder ideals introduced by
Conca in \cite{Co}, and the ideals of adjacent minors introduced by Hoşten and Sullivan
in \cite{HS}. More recently  two new ideals of this kind were introduced: the  binomial edge ideals  by Herzog et al. in \cite{HHHKR}, and independently by Ohtani in \cite{Oh} and polyomino ideals by Qureshi
in \cite{Qu}. A nice survey on these ideals is the book \cite{HHO}.  We recall that polyominoes arise from two-dimensional objects obtained by
joining edge by edge unitary squares, called polyominoes (see \cite{Go}). Over the last few years,
algebraic properties of polyomino ideals have been investigated.
Within them, one of the nicest result is that simple polyominoes, roughly speaking polyominoes without holes, are Cohen-Macaulay domains (see \cite{HM,QSS}).

Moreover, under the assumption that the polyomino $\MP$ is simple and thin, namely it does not contain a square tetromino as subpolyomino, the first and second authors, described the Hilbert series of $S/I_\PP$ (see Section \ref{sec:prel} for the definition of the ring) in terms of the rook complex defined on $\MP$ (see \cite{RR21}). Thanks to this observation, they characterized the Gorenstein simple thin polyominoes. Until now only some other cases of Gorenstein polyominoes are known (see \cite{An,CNU,EHQR,Qu,QRR}).

There are two interesting generalizations of Gorenstein rings: level rings (see \cite{St}) and  pseudo-Gorenstein rings (see \cite{EHHM}). Observing that  a ring  is Gorenstein if and only if the canonical module is a cyclic module, and hence generated in a single degree, the two generalizations naturally arise. In fact, if  one only
requires that the generators of the canonical module  are of the same degree, then the ring is called level,
and if one requires that there is only one generator of least degree, then we call it  pseudo-Gorenstein. A first study on this topic on binomial edge ideals has been carried on by the first and third authors (see \cite{RS23}).

 As stated by Herzog and others (see \cite{EHHM}), since pseudo-Gorenstein rings can be identified by the property that the leading coefficient of the
numerator polynomial of the Hilbert series is equal to 1, pseudo-Gorenstein rings are
much easier accessible than level rings.
This assertion is in particular true for simple thin polyominoes. In fact they can be described by the existence of a unique  configuration of non-attacking rooks of maximum cardinality (see Lemma \ref{lem:pseudo-gorenstein}).

Most of our paper is dedicated to the classification of pseudo-Gorenstein (res. level) simple polyominoes that are paths.

After giving the necessary preliminaries (see Section \ref{sec:prel}), in Section \ref{sec:pseudo} we give a complete description of pseudo-Gorenstein path polyominoes in terms of the non-existence of odd stair, where a stair is a sequence of intervals of length $2$ inside the path (see Theorem \ref{theo:pseudo-gorenstein}). Moreover, this gives us hints for sufficient conditions on  the non-levelness.

To reach the next goal, that is the classification of path polyominoes that are level, we prove that the socle of the ring $S/I_\MP$ modulo some linear forms is generated in the same degree. To this aim, is fundamental to define a nice labelling related to the cells of the path and a Gr\"{o}bner basis induced by it. In particular we prove that this  Gr\"{o}bner basis is quadratic. The labelling and results are in Section \ref{sec:grobner}.

In Section \ref{sec:level}, we describe a linear system of parameters and a ring of dimension 0 such that also in this case the non-attacking rooks defined on $\MP$ play a fundamental role. This gives us a way to classify all level path polyominoes such that their initial ideals are level, and therefore the rings are level, too. Moreover, through the fundamental Lemma \ref{lem:inSoc}, in Example \ref{exam: level but initial is not level}, we describe a path polyomino that is level and whose initial ideal is not level: the stair $S_5$. Through the same Lemma we prove that any path polyomino without stair is level. In the last part of the section by some technical lemmas we describe all pathological stairs, namely the bad stairs, whose existence in the path polyomino make it non-level. These are the ones having length 4,6, or length greater than or equal to 8. Thanks to bad stair we obtain the classification of path polyominoes that are level (see Theorem \ref{theo:level}).

In the last section, we focus on simple thin polyominoes. We share the computation that has been done regarding the classification of Gorenstein, level and pseudo-Gorenstein simple thin polyominoes. The result of the computation is downloadable from \cite{RRSw}. Inspired by this computation and from the results on path polyominoes, we describe the regularity of the pseudo-Gorenstein simple thin polyominoes in relation to their rank (see Theorem \ref{thm:pgrank}), and we also present a conjecture on level simple thin polyominoes.

\vskip 2mm
\noindent
\textbf{Acknowledgment:}
The third author thanks the University of Messina for the hospitality where this project has been started and the National Board for Higher Mathematics (NBHM), Department of Atomic Energy, Government of India, for financial support through Postdoctoral Fellowship.

\section{Polyominoes and Rook complex}\label{sec:prel}
Let $a = (i, j), b = (k, \ell) \in \NN^2$, with $i	\leq k$ and $j\leq\ell$, the set $[a, b]=\{(r,s) \in \NN^2 : i\leq r \leq k \text{ and } j \leq s \leq \ell\}$ is called an \textit{interval} of $\NN^2$. If $i<k$ and $j < \ell$, $[a,b]$ is called a \textit{proper interval}, and the elements $a,b,c,d$ are called corners of $[a,b]$, where $c=(i,\ell)$ and $d=(k,j)$. In particular, $a,b$ are called \textit{diagonal corners} and $c,d$ \textit{anti-diagonal corners} of $[a,b]$. The corner $a$ (resp. $c$) is also called the left lower (resp. upper) corner of $[a,b]$, and $d$ (resp. $b$) is the right lower (resp. upper) corner of $[a,b]$. 
A proper interval of the form $C = [a, a + (1, 1)]$ is called a \textit{cell}. Its vertices $V(C)$ are $a, a+(1,0), a+(0,1), a+(1,1)$. The sets $ \{a,a+(1,0)\}, \{a,a+(0,1)\},\{a+(1,0),a+(1,1)\},$ and $\{a+(0,1),a+(1,1)\}$ are called the \textit{edges} of C.
Let $\MP$ be a finite collection of cells of $\NN^2$, and let $C$ and $D$ be two cells of $\MP$. Then $C$ and $D$ are said to be \textit{connected} if there is a sequence of cells $C = C_1,\ldots, C_m = D$ of $\MP$ such that $C_i\cap C_{i+1}$ is an edge of $C_i$
for $i = 1,\ldots, m - 1$. In addition, if $C_i \neq C_j$ for all $i \neq j$, then $C_1,\dots, C_m$ is called a \textit{path} (connecting $C$ and $D$). A collection of cells $\MP$ is called a \textit{polyomino} if any two cells of $\MP$ are connected. We denote by $V(\MP)=\cup _{C\in \MP} V(C)$ the vertex set of $\MP$. The number of cells of $\MP$ is called the \textit{rank} of $\MP$, and it is denoted by $\rk(\MP)$. 

A polyomino $\mathcal{Q}$ is called a \textit{subpolyomino} of a polyomino $\MP$ if each cell belonging to $\MQ$ belongs to $\MP$, and we write $\MQ\subset \MP$.
A proper interval $[a,b]$ is called an \textit{inner interval} of $\MP$ if all cells of $[a,b]$ belong to $\MP$.
We say that a polyomino $\MP$ is \textit{simple} if for any two cells $C$ and $D$ of $\NN^2$ not belonging to $\MP$, there exists a path $C=C_{1},\dots,C_{m}=D$ such that $C_i \notin \MP$ for any $i=1,\dots,m$. 

We say that two vertices $a, b \in V(\MP)$ are \emph{diagonally opposite}, or simply \emph{opposite}, if they are either diagonal or antidiagonal corners of an inner interval of $\MP$. 

An maximal inner interval $[a,b]$ of $\MP$ with $a=(i,j)$, $b=(k,\ell)$ and either $k-i=1$ or $\ell-j=1$ is identified as a row or column of cells, called \emph{cell interval}.
Let $\PP$ be a simple polyomino. We say that a cell $C$ of $\PP$ is a \emph{leaf} if there exists an edge $\{u,v\}$ of $C$  such that $\{u,v\}\cap V(\PP \setminus \{C\})=\varnothing$. We call the vertices $u$ and $v$ \emph{leaf corners} of $C$. 

We say that a polyomino $\MP$  is \emph{thin} (see \cite{MRR2}, \cite{RR21}) if $\MP$ does not contain the square tetromino (see Figure \ref{fig:square}) as a subpolyomino.
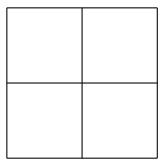
\begin{figure}[H]
\centering
\begin{tikzpicture}
\draw (0,0)--(2,0);
\draw (0,1)--(2,1);
\draw (0,2)--(2,2);

\draw (0,0)--(0,2);
\draw (1,0)--(1,2);
\draw (2,0)--(2,2);
\end{tikzpicture}\caption{The square tetromino}\label{fig:square}
\end{figure}
In a thin polyomino, any maximal interval is a cell interval. For $k\in \mathbb{N}$,
$k$ rooks on the cells of a polyomino $\PP$ are said to be \emph{non-attacking} if they do not lie on the same row or column of cells of $\PP$, pairwise.  The maximum number of non-attacking rooks is called the \emph{rook number} of $\PP$, denoted by $r(\PP)$. We identify the rooks that can be placed on $\PP$ with the cells of $\PP$. The set $\RR_\PP$ of sets of non-attacking rooks is a simplicial complex and it is called \emph{rook complex}. Let $r_k$ be the number of configurations of $k$-non attacking rooks. The polynomial 
\[
r_{\PP}(t)=\sum_{k=0}^{r(\PP)} r_k t^k
\] 
is called the {\em rook polynomial} of $\PP$.

We recall the following notation from \cite{Ro}. Let $\MC=\{I_1,\dots,I_s\}$ be the set of the maximal cell intervals of $\MP$.

\begin{definition}
Let $\MP$ be a polyomino. A subset $\AA \subset \MC$ is called a \emph{partition} of $\MP$ if 
\begin{enumerate}
    \item $\forall I,J \in \AA$ we have $I \cap J = \varnothing$;
    \item $\bigcup_{I \in A} I = \MP$.
\end{enumerate}
\end{definition}

\begin{definition}
An interval  $I=\{C_1,\ldots,C_m\} \in \MC$ is called \emph{embedded} if there exists $F=\{D_1,\ldots , D_m \} \in \RR_\PP$  such that for any $i \in \{1,\ldots, m\}$ the set $\{C_i, D_i\}$ is attacking.
\end{definition}
\begin{remark}\label{rmk:notemb}
Let $I\in \MC$ be a non-embedded interval. Then any facet $F \in \RR_\MP$ is such that $F \cap I \neq \varnothing$.
\end{remark}

 \begin{definition}
Let $\AA$ be a partition of $\MP$. If no interval of $\AA$ is embedded then $\AA$ is called \emph{super partition}. Moreover we call $\MP$ \emph{superpartitionable}
 \end{definition}
With the help of superpartitions one can characterize the polyominoes having a pure rook complex $\RR_\MP$.

 \begin{theorem}[\cite{Ro}, Theorem 3.10]\label{thm:pure}
 Let $\MP$ be a polyomino. The following are equivalent:
 \begin{itemize}
     \item[(1)] $\RR_\MP$ is pure and has dimension $r-1$;
     \item[(2)] $\MP$ admits a super partition $\AA$ with $|\AA|=r$.
 \end{itemize}
 \end{theorem}
We now recall the definition of two important homological invariants which can be computed directly from the Betti table. Let $M$ be a finitely graded $S$-module and $\beta_{i,i+j}(M)$ the graded Betti numbers. The 
 \textit{projective dimension} of $M$ is defined as $\pd(M):=\max\{i : \beta_{i,i+j}(M) \neq 0 \text{ for some } j\}$
 and the \textit{Castelnuovo-Mumford regularity} (or simply, \textit{regularity}) of $M$ is defined as 
 $\reg(M):=\max \{j : \beta_{i,i+j}(M) \neq 0 \text{ for some } i\}.$ 

Let $\MP$ be a polyomino. Let $\mathbb{K}$ be an arbitrary field and $S = \mathbb{K}[x_v \ : \ v \in V(\MP)]$. The binomial $x_a x_b - x_c x_d\in S$ is called an \textit{inner 2-minor} of $\MP$ if $[a,b]$ is an inner interval of $\MP$, where $c,d$ are the anti-diagonal corners of $[a,b]$. We denote by $\MM$ the set of all inner 2-minors of $\MP$. The ideal generated by $\MM$ in $S$ is said to be the \textit{polyomino ideal} of $\MP$ and it is denoted by $I_\MP.$ 
The properties of $I_\MP$ and $S/I_\MP$ arise from combinatorial properties of $\MP$. In \cite{RR21}, the authors prove the following 
\begin{theorem}[Theorem 1.1]
Let $\PP$ be a simple thin polyomino such that the reduced Hilbert-Poincar\'e series of $S/I_\PP$  is
\[
\mathrm{HP}_{S/I_\PP}(t)=\frac{h(t)}{(1-t)^d}.
\]
Then $h(t)$ is the rook polynomial of $\PP$. 
\end{theorem}%
In the same article, the authors introduce a property that is fundamental to characterize Gorenstein simple thin polyominoes
\begin{definition}
Let $\PP$ be a simple thin polyomino.
A cell $C$ of $\PP$ is \emph{single} if there exists a unique maximal inner interval of $\PP$ containing $C$. If any maximal inner interval of $\PP$ has exactly one single cell, we say that $\PP$ has the \emph{S-property}.
\end{definition}
\begin{theorem}
Let $\PP$ be a simple thin polyomino.
Then the following conditions are equivalent:
\begin{itemize}
\item[(1)] $S/I_\PP$ is Gorenstein;
\item[(2)] $\PP$ has the $S$-property.
\end{itemize}
\end{theorem}
We give the definition of the main object of the paper.

\begin{definition}\label{def:path}
	A simple polyomino $\mathcal{P}$ is called a path if $\MP=\{C_1,\dots,C_\ell\}$ such that
	\begin{enumerate}
		\item $C_i\cap C_{i+1}$ is a common edge for all $i=1,\dots,\ell$;
		\item $C_i\neq C_j$ for all $i\neq j$;
		\item For all $i\in \{3,\dots,\ell-2\}$ and $j\notin \{i-2,i-1,i,i+1,i+2\}$, one has $C_i\cap C_j=\emptyset$.
	\end{enumerate}
	We denote a path polyomino by $\PP=C_1 C_2 \cdots C_\ell$.
	If $I_1,\ldots, I_s$ are the maximal intervals of $\PP$, then we denote by $l_k$ the \emph{length} of $I_k$, say $l_k=|I_k|$ for $k \in \{1,\ldots, s\}$. 
\end{definition}

An example of path polyomino is shown in Figure \ref{fig:path}.
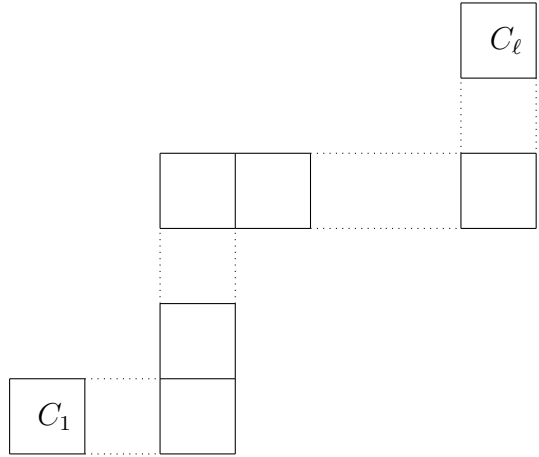
\begin{figure}[H]\label{fig:example}
	\begin{tikzpicture}[scale=1]
		\draw  (-5,-3)-- (-4,-3);
		\draw  (-5,-3)-- (-5,-4);
		\draw  (-5,-4)-- (-4,-4);
		\draw  (-4,-4)-- (-4,-3);
		\draw  (-3,-3)-- (-2,-3);
		\draw  (-2,-3)-- (-2,-4);
		\draw  (-2,-4)-- (-3,-4);
		\draw  (-3,-4)-- (-3,-3);
		\draw  (-3,-2)-- (-3,-3);
		\draw  (-3,-2)-- (-2,-2);
		\draw  (-2,-2)-- (-2,-3);
		\draw  (-3,-1)-- (-3,0);
		\draw  (-3,0)-- (-2,0);
		\draw  (-2,0)-- (-2,-1);
		\draw  (-3,-1)-- (-2,-1);
		\draw  (-2,0)-- (-1,0);
		\draw  (-1,0)-- (-1,-1);
		\draw  (-2,-1)-- (-1,-1);
		\draw  (1,0)-- (2,0);
		\draw  (2,0)-- (2,-1);
		\draw  (2,-1)-- (1,-1);
		\draw  (1,-1)-- (1,0);
		\draw [dotted] (1,0)-- (1,1);
		\draw  (1,1)-- (2,1);
		\draw [dotted] (2,1)-- (2,0);
		\draw[dotted] (-4,-3) --(-3,-3);
		\draw[dotted] (-4,-4) --(-3,-4);
		\draw[dotted] (-3,-2) --(-3,-1);
		\draw[dotted] (-2,-2) --(-2,-1);
		\draw[dotted] (-1,0) --(0,0);
		\draw[dotted] (-1,-1) --(0,-1);
		\draw (1,1) --(1,2);
		\draw (2,1) --(2,2);
		\draw (1,2) --(2,2);
		\draw[dotted] (0,0)-- (1,0);
		\draw[dotted] (0,-1)--(1,-1);
			\draw (-4.4,-3.5) node {$C_1$};
			\draw (1.6,1.5) node {$C_\ell$};
	\end{tikzpicture}
	\caption{A path polyomino $\mathcal{P}$}\label{fig:path}
\end{figure}

We start with the following crucial lemma for any polyomino which will be used in the sequel.

\begin{lemma}\label{lem:F+1}
Let $\MP$ be a polyomino and let $F \in \FF(\RR_\MP)$. If $|F|<r(\MP)$, then there exists $F' \in \FF(\RR_\MP)$ such that $|F'|=|F|+1$.
\end{lemma}
\begin{proof}
Let $|F|=t<r(\MP)$.
Let $G \in \FF(\RR_\MP)$ be such that $|G|=r(\MP)$. Let $A \subset F$ and $B \subset G$ such that
\begin{enumerate}
\item  $A=(F\cap G) \cup \{C \in F: \mbox{ $C$ attacks a unique cell of $G$}\};$	\item $B=(F\cap G) \cup \{D \in G: \mbox{$D$ is attacked by a cell of $A$}\}.$

\end{enumerate} 
We have $|A|=|B|$, therefore let $\widetilde{F}=F\setminus A=\{C \in F :\mbox{ $C$ attacks a two cells of $\widetilde{G}$} \}$, and let $\widetilde{G}=G \setminus B$. We set $\tilde{t}=|\widetilde{F}|$ and $\tilde{d}=|\widetilde{G}|$. It holds $\tilde{d}-\tilde{t}=d-t\geq 1$. We now partition $\widetilde{G}=\widetilde{G}_1 \sqcup \widetilde{G}_2$ in a way such that 
\[
\widetilde{G}_1= \{D \in \widetilde{G}: \mbox{$D$ is attacked by a unique cell of $\widetilde{F}$}\}
\] 
\[
\widetilde{G}_2= \{D \in \widetilde{G}: \mbox{$D$ is attacked by two cells of $\widetilde{F}$}\},
\] 
and let $g_i=|\widetilde{G}_i|$ for $i=1,2$. We observe that since any cell of $\widetilde{F}$ attacks two cells of $\widetilde{G}$, then $2\tilde{t}=\tilde{d}+g_2$, because we count twice the cells of $\widetilde{G}_2$. Hence 
\[
g_1=\tilde{d}-g_{2}=\tilde{d}-(2\tilde{t}-\tilde{d})=2(\tilde{d}-\tilde{t})\geq 2,
\]
that is there are at least two cells of $\widetilde{G}$ attacked by a cell of $\widetilde{F}$. Let $D_1 \in \widetilde{G}_1$ and let $C_1 \in \widetilde{F}$ be the cell attacking it. Let $D_2$ be the other cell of $\widetilde{G}$ attacked by $C_1$. If $D_2 \in \widetilde{G}_1$, then $F'=(F \setminus \{C_1\}) \cup \{D_1,D_2\}$ is the desired facet. Otherwise since $D_2 \in \widetilde{G}_2$, then set $F_1=(F \setminus \{C_1\}) \cup \{D_1\}$, let $C_2$ be the other cell of $\widetilde{F}$ attacking $D_2$ and let $D_3$ be the other cell of $\widetilde{G}$ attacked by $C_2$. Again if $D_3\in \widetilde{G}_1$, then $F'=(F_1 \setminus \{C_2\}) \cup \{D_2,D_3\}$, otherwise set  $F_2=(F_1 \setminus \{C_2\}) \cup \{D_2\}$. By proceeding in this way, since $g_1\geq 2$, we will find a cell $D_{k}$ that is attacked by a cell $C_k$ of the configuration $F_k$ and $C_k$ attacks $D_{k+1}\in  \widetilde{G}_1$, that is we set $F'=(F_k \setminus \{C_k\}) \cup \{D_k,D_{k+1}\}$ and the assertion follows.
\end{proof}

\section{Pseudo-Gorenstein path polyominoes}\label{sec:pseudo}
From now onward, we assume that the polyominoes are simple and thin. We start this section by the following notion that is fundamental for the characterization of pseudo-Gorenstein, and level, path polyominoes.

\begin{definition}\label{def:stair}
Let $\MS$ be a path polyomino as in Definition \ref{def:path} and let $\MC=\{I_1,\dots,I_\lambda\}$ be the set of its maximal cell intervals. Then $\MS$ is called a \textit{stair}, if $\lambda\geq 3$ and $l_i=2$ for all $1< i<\lambda$.  The length of the stair is $\lambda$ and $\MS$ has odd (resp. even) length  if $\lambda$ is odd (resp. even). We denote by $S_{\lambda}$ the stair with $l_1=2=l_{\lambda}$ and by $\tilde{S}_{\lambda}$ a stair  with $l_1>2$ or $l_{\lambda}>2$.
	
\end{definition}

\begin{figure}[H]
	\begin{tikzpicture}[scale=1]
		\draw  (3,0)-- (3,-1);
		\draw  (4,-1)-- (4,0);
		\draw  (4,0)-- (3,0);
		\draw  (3,-1)-- (4,-1);
		\draw  (4,-1)-- (5,-1);
		\draw  (5,0)-- (5,-1);
		\draw  (4,0)-- (5,0);
		\draw  (4,0)-- (4,1);
		\draw  (4,1)-- (5,1);
		\draw  (5,1)-- (5,0);
		\draw  (5,0)-- (6,0);
		\draw  (5,1)-- (6,1);
		\draw  (6,1)-- (6,0);
		\draw  (5,1)-- (5,2);
		\draw  (5,2)-- (6,2);
		\draw  (6,2)-- (6,1);
		\begin{scriptsize}
			\fill  (3,-1) circle (1.5pt);
			\fill  (4,-1) circle (1.5pt);
			\fill  (5,-1) circle (1.5pt);
			\fill  (5,0) circle (1.5pt);
			\fill  (4,0) circle (1.5pt);
			\fill  (6,0) circle (1.5pt);
			\fill  (3,0) circle (1.5pt);
			\fill  (4,1) circle (1.5pt);
			\fill  (5,1) circle (1.5pt);
			\fill  (6,1) circle (1.5pt);
			\fill  (5,2) circle (1.5pt);
			\fill  (6,2) circle (1.5pt);
		\end{scriptsize}
	\end{tikzpicture}
	\caption{The stair $S_4$}\label{fig:stair}
\end{figure}
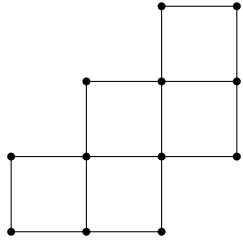

The following is a consequence of \cite{RR21}.
\begin{lemma}\label{lem:pseudo-gorenstein}
	Let $\MP$ be a polyomino. Then $S/I_\MP$ is pseudo-Gorenstein if and only if there exists a unique configuration of non-attacking rooks of maximum cardinality.
\end{lemma}

\begin{remark}\label{rem:pg}
 Let $\PP$ be a simple thin polyomino such that $S/I_\MP$ is pseudo-Gorenstein. Then any interval has at most one single cell. Moreover, 
	let $F\in\FF(\RR_\MP)$ be the unique facet of maximum cardinality of the pseudo-Gorenstein ring  $S/I_\MP$. Then $F$ contains all single cells of $\MP$. 
\end{remark}
\begin{proof} Suppose $D_i$ is a single cell contained in the interval $I_i$ and let $C_i\in F$ with $C_i\in I_i$. Then 
	\[
	(F\setminus \{C_i\})\cup \{D_i\}
	\]
	is a facet. That is $C_i=D_i$.
\end{proof}

\begin{theorem}\label{theo:pseudo-gorenstein}
	Let $\MP$ be a path polyomino with $\MC=\{I_1,I_2,\ldots ,I_s\}$. Then $S/I_\MP$ is pseudo-Gorenstein  if and only if either $\MP$ is a cell or the following conditions hold:
	\begin{enumerate}
		\item $l_1=l_s=2$ and $l_k\leq 3$ for all $2\leq k\leq s-1$;
		\item $\MP$ does not have odd stairs.
	\end{enumerate}
\end{theorem}
\begin{proof}
	The case $\MP$ is a cell is obvious being a principal ideal. Hence from now on we assume $\MP$ is not a cell.
	
	$\Rightarrow )$ Since $\MP$ is not a cell, then $s>1$. Moreover, by Remark \ref{rem:pg}, every interval has at most one single cell, and  since $\MP$ is a path,  then $l_1=l_s=2$ and $l_k\leq 3$. That is (1) holds. Now, assume $F \in \FF(\RR_\MP)$ is the unique facet of maximal cardinality.
	Let $I_{k+1}, I_{k+2},\ldots, I_{k+\ell}$  be intervals of an odd stair, with first interval $I_{k+1}$, last interval $I_{k+\ell}$ and $\ell$ odd. Then we may assume $\{C_k,C_{k+1}\}\subseteq I_{k+1}$, $\{C_{k+2},C_{k+3}\}=I_{k+3}$, $\{C_{k+4},C_{k+5}\}= I_{k+5}$,$\ldots$, $\{C_{k+\ell-1},C_{k+\ell}\}\subseteq I_{k+\ell}$, and $\MP'=C_k C_{k+1}\cdots C_{k+\ell}$ is a subpath of $\MP$. Since $F$ is unique, then by Remark \ref{rem:pg} we have that the single cell $C_k$ of $I_{k+1}$, and $C_{k+\ell}$ of $I_{k+\ell}$ are in $F$. Now, 
	\[
	G=(F\setminus \MP')\cup \{C_k,C_{k+2},\ldots, C_{k+\ell-1}\}
	\]
	has the same cardinality of $F$ and it is a facet, too. Hence $F$ is not unique. That is (2) holds, too.
	
	$\Leftarrow)$ By the Lemma \ref{lem:pseudo-gorenstein}, we need to prove that there exists a unique configuration of non-attacking rooks of maximum cardinality. We claim that if (1) and (2) are satisfied then $\MP=C_1\cdots C_{2m+1}$ and the unique maximal facet is $F=\{C_1,C_3,\ldots,C_{2m+1}\}$. Suppose there are not even stairs, namely $\MP$ does not contain any stair.  Then $\MP$ is Gorenstein and $F$ has only single cells: $C_1$ of the interval $I_1=\{C_1,C_2\}$, $C_3$ of $I_2=\{C_2,C_3,C_4\}$, and so on, ending with $C_{2m+1}$ in $I_s=\{C_{2m},C_{2m+1}\}$. Now, suppose that $\MP$ has at least an even stair and without loss of generality that the first even stair is in the first $\ell$ intervals that is $I_1=\{C_1,C_2\}$, $I_2=\{C_2,C_3\}$, and so on until the last one that is $I_\ell=\{C_\ell,C_{\ell+1},C_{\ell+2}\}$ with $\ell$ even. Since $F$ is maximal, then $C_{\ell+1}$ must belong to $F$.  Moreover, $F$ having maximum cardinality contains $\{C_1,C_3,\ldots, C_{\ell-1}, C_{\ell+1}\}$. Using induction on the number of even stairs the assertion easily follows.

\end{proof}

By the proof of Theorem \ref{theo:pseudo-gorenstein} we obtain the following
\begin{corollary}\label{cor:pseudo-gorenstein}
	Let $\MP$ be a path polyomino such that  $S/I_\MP$ is pseudo-Gorenstein with $\MP=C_1C_2\cdots C_{\ell}$. Then the unique facet of maximum cardinality is $$F=\{C_1,C_3,\ldots, C_{\ell}\}$$ and 
	\begin{enumerate}
		\item $\ell=2r-1$;
		\item $r(\MP)=r$.
	\end{enumerate}
\end{corollary}

\section{Gr\"obner basis of path polyominoes}\label{sec:grobner}

We start this section by defining a labelling that induces in a natural way a quadratic Gr\"obner basis to the polyomino ideal.

\begin{notation}\label{not:path-poly}
	Let $\mathcal{P}$ be a path polyomino with cells $C_1,C_2,\ldots , C_{\ell}$. For any $i\in \{1,\ldots , \ell \}$, we call $\MP_i$ the subpath on the cells $C_1,C_2, \ldots , C_i$. Then, we relabel the vertices of $V(\MP)$ as $\{a,a_1,b_1,\ldots, a_\ell, b_\ell, b \}$ such that  
\begin{itemize}
\item $a,a_1$  are leaf corners of $C_1$, and $b_1$ is the opposite corner of $a_1$ in $C_1$;
\item for any $i\in \{2,\ldots, \ell\}$, $a_i$ is leaf corner of $C_{i-1}$ in $\MP_{i-1}$, and $b_i$ is the leaf corner of $C_i$ in $\MP_i$ opposite to $a_i$.
\item $b$ is the leaf corner of $C_\ell$ different from $b_\ell$.
\end{itemize} 
From now on, we consider the graded reverse lexicographic order $<$ on $S$ such that $x_{b_1}>x_{a_1}>\cdots >x_{b_\ell}>x_{a_\ell}>x_{b}>x_{a}$.
\end{notation} 

\begin{figure}[H]
\centering
\begin{tikzpicture}
\draw (0,0)--(2,0);
\draw (3,0)--(5,0);
\draw (0,1)--(5,1);
\draw (1,2)--(4,2);

\draw (0,0)--(0,1);
\draw (1,0)--(1,2);
\draw (2,0)--(2,2);
\draw (3,0)--(3,2);
\draw (4,0)--(4,2);
\draw (5,0)--(5,1);

\filldraw (0,0)  circle (\rad) node [anchor=north east]{$a$};
\filldraw (0,1)  circle (\rad) node [anchor=south east]{$a_1$};
\filldraw (1,0)  circle (\rad) node [anchor=north]{$b_1$};
\filldraw (2,0)  circle (\rad) node [anchor=north]{$b_2$};
\filldraw (1,1)  circle (\rad) node [anchor=south east]{$a_2$};
\filldraw (1,2)  circle (\rad) node [anchor=south east]{$b_3$};
\filldraw (2,1)  circle (\rad) node [anchor=north west]{$a_3$};
\filldraw (2,2)  circle (\rad) node [anchor=south]{$a_4$};
\filldraw (3,2)  circle (\rad) node [anchor=south]{$a_5$};
\filldraw (4,2)  circle (\rad) node [anchor=south]{$a_6$};
\filldraw (3,1)  circle (\rad) node [anchor=north west]{$b_4$};
\filldraw (4,1)  circle (\rad) node [anchor=south west]{$b_5$};
\filldraw (3,0)  circle (\rad) node [anchor=north]{$b_6$};
\filldraw (4,0)  circle (\rad) node [anchor=north]{$a_7$};
\filldraw (5,0)  circle (\rad) node [anchor=north west]{$b$};
\filldraw (5,1)  circle (\rad) node [anchor=south west]{$b_7$};
\end{tikzpicture}
\caption{Path labelled as in Notation \ref{not:path-poly} }\label{fig:general-path}
\end{figure}
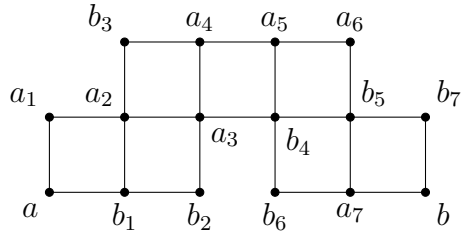

\begin{remark}\label{rem:Ci+1}
From Notation \ref{not:path-poly}, if $C_i$ has leaf corners $a_{i+1}$ and $b_{i}$, and other corners $a_i$ and $b_{i-1}$ in the polyomino $\PP_i$, then the leaf corners of $C_{i+1}$ are $a_{i+2}$ and $b_{i+1}$ with the other corners either $\{a_{i},a_{i+1}\}$ or $\{a_{i+1},b_i\}$ or $\{b_{i-1},b_i\}$ (see Figure \ref{fig:Ci}).
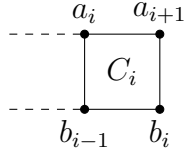
\begin{figure}[H]
\begin{tikzpicture}
\draw[dashed] (0,0)--(1,0);
\draw[dashed] (0,1)--(1,1);
\draw (1,0)--(2,0);
\draw (1,1)--(2,1);

\draw (1,0)--(1,1);
\draw (2,0)--(2,1);

\node at (1.5,0.5){$C_i$};
\filldraw (1,0)  circle (\rad) node [anchor=north]{$b_{i-1}$};
\filldraw (2,0)  circle (\rad) node [anchor=north]{$b_i$};
\filldraw (1,1)  circle (\rad) node [anchor=south]{$a_{i}$};
\filldraw (2,1)  circle (\rad) node [anchor=south]{$a_{i+1}$};

\end{tikzpicture}
\caption{The labelling of the cell $C_i$}\label{fig:Ci}
\end{figure}
\end{remark}

\begin{definition}
We say that a path polyomino  $\MP=C_1 C_2 \cdots C_\ell$ has a \emph{change of direction} at $C_{i}$ with $1<i<\ell$, if $C_{i-1}\cap C_{i+1}\neq \varnothing$.
\end{definition}
\begin{lemma}
Let $\MP$ be a path polyomino as in Notation \ref{not:path-poly}. Then $\MM$, the set of inner 2-minors of $\MP$ forms a reduced Gr\"obner basis of $I_\MP$.
\end{lemma}
\begin{proof}
Let $f,g \in \MM$ be such that $\gcd(\ini(f),\ini(g))\neq 1$. Let $f=f^+ -f^-$ and $g=g^+ - g^-$ with $f^+=\ini(f)$ and $g^+=\ini(g)$. We divide two cases 
\begin{itemize}
\item[1)] $\gcd(f^-,g^-)\neq 1$;
\item[2)] $\gcd(f^-,g^-)= 1$;
\end{itemize}
In case 1), $f$ and $g$ are inner 2-minors contained in the same maximal interval, in particular the two intervals agree on an edge, and by definition. their $S$-polynomial reduces to 0.\\
In case 2), let $I_1$ and $I_2$ be the inner intervals associated to $f$ and $g$ respectively, clearly $u \in I_1\cap I_2$  with $\gcd(f^+,g^+)= x_u$. We claim $|I_1 \cap I_2|>1$. In fact, if $I_1 \cap I_2=\{u\}$, then $u$ is corner of a cell $C_i$ such that there is a change of direction at $C_i$, and $C_{i-1}$ is a cell of $I_1$ and $C_{i+1}$ is a cell of $I_{2}$. We say that $C_{i-1}$ has diagonal corners $u,w$ and antidiagonal corners $v,t$, while $C_{i+1}$ has diagonal corners $u,p$ and antidiagonal corners $z,q$ with $v,u,z$ lying on the same edge intervals, and $u,t,q \in C_{i}$. Let $c$ be the fourth corner of $C_{i}$. We claim that $p=a_{i}+2$. We divide into two cases:
\begin{itemize}
\item[i)] $a_i=t$;
\item[ii)] $a_i=v$.
\end{itemize}
In case i), we have that $z=b_i$ because it is the corner of $C_i$ opposite to $t$, hence $c=a_{i+1}$, $q=b_{i+1}$ and $p=a_{i+2}$. \\
In case ii), we have that $c=b_i$ because it is the corner of $C_i$ opposite to $v$, hence $z=a_{i+1}$, $q=b_{i+1}$ and $p=a_{i+2}$. \\
In both cases, we have proved  $p=a_{i+2}$, hence the claim follows, moreover the latter implies that $u$ is opposite to $a_{i+2}$ in $I_2$ that is the least variable in $I_2$ and $x_u \not | g^+$ (See Figure \ref{fig:twoturns}).


\begin{figure}[H]
\begin{tikzpicture}
\draw (0,0)--(0,2);
\draw (1,0)--(1,2);
\draw (2,1)--(2,2);

\draw (0,0)--(1,0);
\draw (0,1)--(2,1);
\draw (0,2)--(2,2);

\filldraw (1,0) circle (\rad) node [anchor=north west]{$v$};
\filldraw (1,1) circle (\rad) node [anchor=north west]{$u$};
\filldraw (1,2) circle (\rad) node [anchor=south west]{$z$};

\draw (0,0)--(1,1);
\draw (1,0)--(0,2);
\draw (1,2)--(2,1);
\end{tikzpicture}
\caption{The diagonals represent the leading monomials}\label{fig:twoturns}
\end{figure}
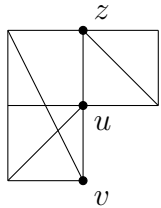

The claim tells us that $|I_1 \cap I_2|>1$, that is $I_1$ and $I_2$ have an edge in common, say $\{u,v\}$. Moreover let $f=x_u x_t - x_v x_w$ and $g=x_u x_z-x_p x_q$ with $p$ lying on the same edge of $u$. Without loss of generality, we assume that $u,v$ are corners of the interval $I_1$, and that $\{u,v\}$ is edge of the cell  $C_i=\{u,v,p,c\}$  in the interval $I_2$, and there is a change of direction at $C_i$.
We have, $I_2\setminus \{ C_i \}$ has associated inner $2$-minor $g'=x_vx_z-x_cx_q$ and $I_1 \cup \{C_i \}$  has associated inner $2$-minor $f'=x_tx_p-x_wx_c$. By taking the $S$-polynomial, we get 
\[
S(f,g)= x_z x_v x_w - x_t x_p x_q.
\]

Since $w,t,q,z$ are extremal variables and $x_t| \ini(f) $ and $x_z |\ini(g)$, then we should identify the least variable of between $x_q$ and $x_w$.

We divide into two cases that depend on the orientation of the polyomino:
\begin{enumerate}
\item $\{u,v\}$ is edge of $C_{i-1}$;
\item $\{u,v\}$ is edge of $C_{i+1}$.
\end{enumerate}
In case (1),  $z,q$ is edge of some $C_k$ with $k>i$, and by construction we have $q=a_{k+1}$, $\ini(S(f,g))=x_z x_v x_w $, and $S(f,g)$ can be reduced modulo $g'$ to obtain
\[
x_z x_v x_w - x_t x_p x_q- x_w g'=x_wx_cx_q- x_t x_p x_q=x_qf' \to 0.
\]
In case (2), $w,t$ is edge of some $C_k$ with $k>i$, and by construction we have $w=a_{k+1}$, $\ini(S(f,g))=x_t x_p x_q $, and $S(f,g)$ can be reduced modulo $f'$ to obtain
\[
x_t x_p x_q-x_z x_v x_w - x_q f'=x_wx_cx_q- x_z x_v x_w=x_wg' \to 0.
\]

\end{proof}

\begin{lemma}\label{lem:inI}
Let $\MP$ be a path polyomino as in Notation \ref{not:path-poly}. Then
\begin{enumerate}
\item $x_{a_i}x_{a_j}\notin \ini(I_\MP)$ for any $i,j$;
\item $x_{a_i}x_{b_j} \in \ini(I_\MP)  \Rightarrow i\leq j$;
\item $x_{b_i}x_{b_j}\in \ini(I_\MP)$ with $i< j$ $\Rightarrow$ there is a change of direction at $C_i$; 
\item Let $f=f^+-f^-$ be a generator of $I_\MP$. If $f^+=x_ux_{b_i}$ with either $x_u>x_{b_{i}}$ or $u=a_i$, then $x_{a_{i+1}}|f^-$.
\end{enumerate}
\end{lemma}

\begin{proof}
$(1).$ Without loss of generality assume $j<i$ that is $x_{a_j}>x_{a_i}$. Let $I$ be an inner interval of $\MP$ having opposite corners $a_i$ and $a_j$ and $u,v \in V(\MP)$. From Notation \ref{not:path-poly}, we have that $i\geq j+2$ because $a_{j}$ and $a_{j+1}$ lie on the same edge. If $I$ is only one cell, then $i=j+2$ and $\{u,v\}=\{b_{j+1},a_{j+1}\}$. If $I$ has cells $D_1,\ldots, D_k$, then $u$ and $a_i$ are leaf corners of $D_k$, hence $u=b_{i-1}$. That is $x_u,x_v>x_{a_i}$ and $x_{a_i}x_{a_j}$ is not a leading monomial.\\
$(2).$ If $i>j$, by similar arguments to (1), one can show that $x_{a_i}x_{b_j}\notin \ini(I_\MP)$.\\
$(3).$ If $x_{b_i}x_{b_j}\in \ini(I_\MP)$ with $i< j$, since $b_j$ is opposite to both $b_i$ and $a_j$, then $a_j$ is on the same edge interval of $b_i$. Since $a_{j}$ is the leaf corner of $C_{j-1}$ and the other one is $b_{j-1}$, then $b_i$ is also opposite to $b_{j-1}$. By proceeding in this way, we have that the cell $C_{i+1}$ has opposite corners $b_{i}$ and $b_{i+1}$, and $v,a_{i+2}$ for some $v \in V(\MP)$. By definition, $b_i$ must be a corner of $C_i$, and from Remark \ref{rem:Ci+1} $a_{i+2} \notin C_{i}$, hence ${b,v}$ is an edge of $C_i$, by definition $a_{i+1}$ is a leaf corner of $C_i$, it is opposite to $b_{i+1}$ and hence is opposite to $v$, that is $v \in C_{i-1}$. \\
$(4)$. By construction, $b_i$ is the leaf corner of $C_i$. Since $u$ and $b_i$ are opposite corners, then $a_{i+1}$ lies on the same edge intervals of $u$ and $b_i$, hence $x_{a_{i+1}}|f^-$.
\end{proof}

\section{Levelness of path polyominoes}\label{sec:level}
For rest of the article,  let $\MP=C_1 C_2 \cdots C_\ell$ be a path polyomino with $\rk(\MP)=\ell$ and $\mathcal{C}=\{I_1,\dots,I_s\}$ be its set of the maximal cell intervals. By combining \cite[Theorem 2.1]{HM} and \cite[Corollary 3.3]{HQS}, we have the following:
\begin{lemma}\label{lem:dim-of-poly}
	Let $\MP$ be a simple polyomino and $I_{\MP}$ be the polyomino ideal. Then $S/I_{\MP}$ is a Cohen-Macaulay domain with $\dim(S/I_{\MP})=|V(\MP)|-\rk(\MP)$.
\end{lemma}
Let $\MP$ be a path polyomino. Since a path polyomino is a simple one, by Lemma \ref{lem:dim-of-poly},
$S/I_{\MP}$ is a Cohen-Macaulay domain. We now study the levelness of $S/I_{\MP}$, by the socle of a ring, namely $R/J_\MP$, that has a clear connection with the combinatorial properties of $S/I_{\MP}$. The following notation will be used from now on.
\begin{notation}\label{not:R-and-J_P}
	Let $L=(x_a,x_{a_1}-x_{b_1}, \ldots, x_{a_\ell}-x_{b_\ell}, x_b)\subseteq S$. Let $\phi:S \to R=\KK[\y_1,\dots,\y_\ell]$ be the map such that $\phi(x_a)=\phi(x_{b})=0$ and $\phi(x_{a_i})=\phi(x_{b_i})=\y_i$ for $1\leq i\leq \ell$. Then it can be noted that $S/(I_{\MP}+L)\simeq R/J_{\MP}$, where $J_{\MP}$ is obtained by replacing $x_{a_i}$ and $x_{b_i}$ by $\y_i$ in the generating set of $I_{\MP}$. Also, let ``$<$" denote the graded reverse lexicographic order in $R$ induced by $\y_1>\dots>\y_\ell$. 
\end{notation}
\begin{proposition}\label{prop:regseq}
Let $\MP$ be a path polyomino as in Notation \ref{not:path-poly}. Then 
\[
x_a,x_{a_1}-x_{b_1}, \ldots, x_{a_\ell}-x_{b_\ell}, x_b
\]
is a linear system of parameters for $S/I_\MP$.
\end{proposition}
\begin{proof}
Let $L=(x_a,x_{a_1}-x_{b_1}, \ldots, x_{a_\ell}-x_{b_\ell}, x_b)$, and let $R$ and $J_{\MP}$ be as in Notation \ref{not:R-and-J_P}. Let us consider the graded reverse lexicographic order ``$<$" in $R$ induced by $\y_1>\dots>\y_\ell$. We claim that $\y_i^2 \in \ini(J_\MP)$ for all $i=1,\ldots,\ell$. From Notation \ref{not:path-poly}, we have that for any $i=1,\ldots, \ell$, $a_i$ and $b_i$ are opposite corners, and $a_{i+1}$ is on the same edge intervals of both $b_i$ and $a_i$. Hence, there exists an inner $2$-minor in $I_\MP$ of the form $x_{a_i}x_{b_i}-x_{a_{i+1}}x_v$ for some $v \in V(\MP)$. The image in $R/J_\MP$ of such a binomial is $\y_{i}^2-\y_{i+1}\y_{j}$ for some $j$. We have that $\y_i^2 \in \ini(J_\MP)$, that is the ring $R/J_\MP$ is $0$-dimensional and hence the assertion follows. 
\end{proof}

\begin{proposition}
	Let $\mathcal{P}$ be a path polyomino. Then the generators of $J_{\MP}$ form a Gr\"{o}bner basis. Moreover,
	\[
	\ini(J_{\MP})=(\y_i\y_j: \exists \ k\in[s], C_i,C_j\in I_k ).
	\]
\end{proposition}
\begin{proof}
Let $f'=\phi(f),g'=\phi (g)$ be two generators of $J_\MP$ with $f,g\in I_\MP$ and $f=f^+-f^-,g=g^+-g^-$, where the map $\phi$ is defined in Notation \ref{not:R-and-J_P}. \\
We have to consider $S(f',g')$ in the cases $\gcd(\ini(f'),\ini(g')) \neq 1$, that is $\deg(\gcd(\ini(f'),\ini(g'))) \in \{1,2\}$.\\
If $\deg(\gcd(\ini(f') \ini(g')))=2$, namely $\ini(f')=\ini(g')=\y_i\y_j$ with $i<j$, then from Lemma \ref{lem:inI} $x_{b_j}|\ini(f)$ and $x_{b_{j}}|\ini(g)$, that is without loss of generality we assume  $\ini(f)=x_{a_i}x_{b_j}$ and $\ini(g)=x_{b_{i}}x_{b_j}$ but this cannot happen because $a_i$ and $b_i$ are opposite corners and the polyomino is thin.\\
If $\gcd(\ini(f'),\ini(g'))=\phi(\gcd(\ini(f),\ini(g)))$, then $S(f',g')=\phi(S(f,g))\rightarrow 0$.\\
If $\gcd(\ini(f'),\ini(g'))\neq\phi(\gcd(\ini(f),\ini(g)))=1$, then $x_{a_i}|\ini(f)$ and $x_{b_i}|\ini(g)$. That is $\ini(f)=x_{a_i}x_{b_j}$ for some $j> i$ (if $j=i$ we are in the previous case) and $\ini(g)=x_u x_{b_{i}}$ with $u\in V(\MP)$. From Lemma \ref{lem:inI}.(4), we get that 
\[
f=x_{a_i} x_{b_j} - x_v x_{a_{j+1}},
\]
for some $v$ in $V(\MP)$ with $x_v>x_{b_i},x_{a_i}$. One can observe that since $a_i$ and $b_j$ are opposite corners as well as $a_i$ and $b_i$, then there is no change of direction at $C_i$, and it can not happen that $u=b_k$ with $k>i$. Moreover, from Lemma \ref{lem:inI}.(1), also the case $u=a_{k}$ with $k>i$ is not possible. Hence, we can only have $u=a_k$ or $u=b_k$ for some $k < i$. 
From Lemma \ref{lem:inI}.(4), we get that $x_{a_{i+1}} | g^-$ and
\[
g=x_u x_{b_{i}} - x_w x_{a_{i+1}},
\]
for some $w$ in $V(\MP)$.
By construction, $b_{j}$ lies on the same edge interval of $b_{i}$ as well as $u$ lies on the same edge interval of $a_i$, hence from $j>i$ we get that $h_1=x_{a_{i+1}} x_{b_j}- x_{b_i} x_{a_{j+1}}$ and $h_2=x_{u} x_{v}- x_{w} x_{a_{i}}$ are generators of $I_\MP$. We have 
\[
f'=\y_i\y_j - \phi(x_v)\y_{j+1} \ \  g'=\phi(x_u) \y_{i} - \phi(x_w) \y_{i+1}
\]
\[
h_1'=\phi(h_1)= \y_{i+1}\y_j - \y_i\y_{j+1},\ \ \ h_2'=\phi(h_2)= \phi(x_{u}) \phi(x_{v}) - \phi(x_{w})\y_{i}
\]
hence
\[
S(f',g')=\y_j\phi(x_w) \y_{i+1} - \phi(x_u)\phi(x_v) \y_{j+1}
\]
by reducing modulo $h'_1$ we get 
\[
S(f',g')=\y_i \y_{j+1} \phi(x_w)- \phi(x_u)\phi(x_v) \y_{j+1}=\y_{j+1}h_2'.
\]

\end{proof}

\begin{remark}\label{rem:CiCj}
	Let $\mathcal{P}$ be a path polyomino. Then $\y_i\y_j \in \ini(J_\MP)$ with $j\leq i$ if and only if the cells $C_i$ and $C_j$ are attacking.
\end{remark}

\begin{corollary}\label{cor:standard-monomial-R/J_P}
	The standard monomials in $R/J_{\MP}$ are the squarefree monomials $$u=\y_{i_1}\y_{i_2}\cdots \y_{i_{s'}} \text{ with }
	i_1<i_2<\cdots< i_{s'},$$
	where $s'$-rooks are placed in the non-attacking cells $C_{i_1},\dots,C_{i_{s'}}$ of $\mathcal{P}$.
\end{corollary}
\begin{proof}
	If two rooks are placed in the non-attacking cells $C_{i_j}$, $C_{i_{k}}$ with $i_j\leq i_k$, then $C_{i_j}$ and $C_{i_{k}}$ do not belong to the same interval. 
	By Remark \ref{rem:CiCj}, $\y_{i_j}\y_{i_{k}}\notin \ini(J_{\MP})$ for $1\leq j\leq s'-1$. Therefore, $u\notin \ini(J_{\MP})$. 
	
	Suppose two rooks are placed in the cells $C_{i_j},$ $C_{i_{k}}$ with $i_j\leq i_k$ so that they can attack each other and $\y_{i_j} \y_{i_{k}}$ divides our monomial. Then there exist $m$ such that  $i_j,i_{k}\in I_m$. Then it is reducible by 
	\[
	\y_{i_j}\y_{i_{k}}-\y_{i_j-1}\y_{i_{k}+1},
	\]
	hence our monomial is not in the standard form.
	Contradiction.	
	
\end{proof}

Now we write all the monomial generators of $J_{\mathcal{P}}$ which come from the first cell interval and last cell interval. Let $J_0$ and $J_s$ be the subideals of $J_{\mathcal{P}}$ generated by monomials coming from the first and last cell intervals respectively. Then 
$$J_0=(\y_i\y_j:1\leq i\leq j\leq l_1-i+1) \text{ and }$$
$$J_s=(\y_i\y_j:\ell\geq i\geq j\geq 2\ell-i-(l_s-1)).$$

We also recall that the set of configurations of pairwise non-attacking rooks is a simplicial complex denoted by $\RR_\MP$. Moreover for $F \in \FF(\RR_\MP)$ we set 
\[
\y_{F}=\prod\limits_{C_{i} \in F} \y_i
\]

\begin{lemma}\label{lem:generators-of-inSocle}
Let $\MP$ be a path polyomino. Then 
\[
\soc(R/\ini(J_\MP))=(\y_{F} \ | \ F \in \FF(\RR_\MP)).
\]
\end{lemma}
\begin{proof}
The socle of a monomial ideal is a monomial ideal. Let $u$ be a monomial of $R \setminus J_\MP$. Since $\y_i^2 \in \ini (J_\MP)$, then $u$ is squarefree and $u=\y_F$ for some $F \in \RR_\MP$. Assume that $F$ is not maximal, then there exists $F' \in \FF(\RR_{\MP})$ such that $F \subset F'$ and let $C_{i}\in F' \setminus F$. Since $C_i$ attacks no cell in $F$, then $u \y_{i} \notin \ini(J_\MP)$, contradicting the hypothesis.\\
Conversely, assume that $F \in \FF(\RR_{\MP})$ and $u=y_F$. Now we prove that for all $i\in \{1,\ldots, \rk \MP \}$, $u \y_{i} \in \ini (J_\MP)$. Fix $i \in \{1,\ldots, \rk \MP\}$. Since $F$ is maximal, then the cell $C_i$ is attacked by a cell $C_j$ in $F$, hence it follows from Remark \ref{rem:CiCj} that $\y_i\y_j \in \ini (J_\MP)$.  
\end{proof}
It is interesting to observe that the facets of maximum cardinality of $\FF(\RR_\MP)$ induces also maximum degree elements (monomials) in $\soc(R/J_\MP)$, as pointed out in Corollary \ref{cor:maxdeg}.

The following is an immediate consequence of Lemma \ref{lem:F+1} and Lemma \ref{lem:generators-of-inSocle} as $\beta_{p,p+i}(R/I)$ is the number of elements of degree $i$ in $\soc(R/I)$ for any ideal $I\subseteq R$, where $p=\pd(R/I)$.
\begin{corollary}
Let $\MP$ be a path polyomino $p=\rk(\MP)$ and $t< r(\MP)$. If $\beta_{p,p+t}(R/\ini (J_\MP))\neq 0$ then $\beta_{p,p+t+1}(R/\ini (J_\MP))\neq 0$.
\end{corollary}

\begin{corollary}\label{cor:LevInJ}
Let $\MP$ be a path polyomino. Then $R/\ini (J_\MP)$ is level if and only if $\RR_\MP$ is pure.
\end{corollary}

By using Theorem \ref{thm:pure}, we obtain the following characterization of polyominoes having $R/\ini (J_\MP)$ a level.
\begin{theorem}
Let $\MP$ be a path polyomino with $\MC=\{I_1,I_2,\ldots ,I_s\}$. Then $R/\ini (J_\MP)$ is level with $r(\MP)=d$ if and only if the followings hold
\begin{enumerate}
\item $s=2d-1$.
\item for any $2\leq k\leq s-1$ we have 
\[
\begin{cases}
l_k > 2 &\mbox{ if $k$ is odd}\\ 
l_k = 2 &\mbox{ if $k$ is even}\\ 
\end{cases}
\]
\end{enumerate}
\end{theorem}
\begin{proof}
According to Corollary \ref{cor:LevInJ}, $R/\ini (J_\MP)$ is level if and only if $\RR_\MP$ is pure, that is if and only if $\MP$ admits a super partition $\AA$. 
We proceed by induction on $d$. We start from $d=2$.  $\AA$ contains two intervals, since $I_1 \in \AA$, then $I_2 \notin \AA$ and $s=3$. If $I_2$ contains single cell then $I_2 \in \AA$, hence $|I_2|=2$.
We assume $d > 2$ and the thesis holds true for $d-1$. From similar arguments, one gets that $|I_2|=2$ and that the polyomino $\MP'=\MP\setminus I_1$ is super partitionable, with partition $\AA'=\AA\setminus I_1$. Hence $s-2=2(d-1)-1$, that is $s=2d-1$ and $l_k > 2$ for an odd $k \in \{ 4,\ldots, s-1 \}$ and $l_k = 2$ for even $k \in \{4,\ldots, s-1\}$. We prove that $l_3 > 2$. If $l_3=2$, then $I_3=[C_{i},C_{i+1}]$ and $C_{i-1} \in I_1$, $C_{i+2} \in I_5$ with $C_{i-1}$ (resp. $C_{i+2}$) attacking $C_{i}$ (resp. $C_{i+1}$). That is $I_3$ is an embedded interval. Contradiction.
\end{proof}

We now study the levelness of $R/J_\MP$ with the help of the following
\begin{lemma}\label{lem:inSoc}
Let $I \subseteq R$ be an ideal with $\dim(R/I)=0$. Then for any monomial ordering $<$, one has
\[
\ini (\soc(R/I)) \subset \soc(R/\ini(I)). 
\]
\end{lemma}
\begin{proof}
Let $u \in \ini (\soc(R/I)) $, then there exists $g$ such that $f=u+g$ and $f\y_{i} \in I$ for all $1\leq i\leq n$. Then $u\y_{i} \in \ini(I)$ for all $1\leq i\leq n$, hence $u \in \soc(R/\ini(I))$.
\end{proof}

\begin{lemma}\label{lem:FF'}
Let $\MP$ be a path polyomino, let $F$ be a facet of $\RR_\MP$. If there exists $k$ such that $C_{k} \in F$ and 
\[
F'=(F \setminus \{C_{k}\}) \cup \{C_{k-1},C_{k+1}\}
\]
is a facet of $\RR_\MP$, then $\y_F \notin \ini(\soc(R/J_\MP))$.
\end{lemma}
\begin{proof}
We consider $\y_k \y_F$. The relation $\y_{k}^2-\y_{k-1}\y_{k+1} \in J_\MP$, that is $\y_k \y_F$ reduces to $\y_{F'}$. By Corollary \ref{cor:standard-monomial-R/J_P} $y_{F'} \notin \ini(J_\MP)$.  Hence $\y_{F} \notin \ini(\soc(R/J_\MP))$.
\end{proof}

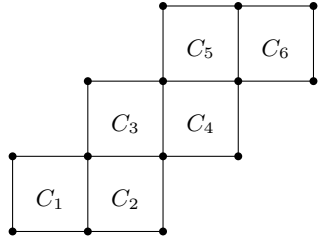
\begin{figure}[H]
	\begin{tikzpicture}[scale=1]
		\draw  (3,0)-- (3,-1);
		\draw  (4,-1)-- (4,0);
		\draw  (4,0)-- (3,0);
		\draw  (3,-1)-- (4,-1);
		\draw  (4,-1)-- (5,-1);
		\draw  (5,0)-- (5,-1);
		\draw  (4,0)-- (5,0);
		\draw  (4,0)-- (4,1);
		\draw  (4,1)-- (5,1);
		\draw  (5,1)-- (5,0);
		\draw  (5,0)-- (6,0);
		\draw  (5,1)-- (6,1);
		\draw  (6,1)-- (6,0);
		\draw  (5,1)-- (5,2);
		\draw  (5,2)-- (6,2);
		\draw  (6,2)-- (6,1);
		\draw  (6,2)-- (7,2);
		\draw  (6,1)-- (7,1);
		\draw  (7,2)-- (7,1);

		\begin{scriptsize}
			\fill  (3,-1) circle (1.5pt);
			\draw (3.50,-0.57) node {$C_1$};
			\fill  (4,-1) circle (1.5pt);
			\draw (4.50,-0.57) node {$C_2$};
			\fill  (5,-1) circle (1.5pt);
			\draw (4.50, 0.43) node {$C_3$};
			\fill  (5,0) circle (1.5pt);
			\draw (5.50, 0.43) node {$C_4$};
			\fill  (4,0) circle (1.5pt);
			\draw (5.50,1.43) node {$C_5$};
			\fill  (6,0) circle (1.5pt);
			\draw (6.50,1.43) node {$C_6$};
			\fill  (3,0) circle (1.5pt);
			\fill  (4,1) circle (1.5pt);
			\fill  (5,1) circle (1.5pt);
			\fill  (6,1) circle (1.5pt);
			\fill  (5,2) circle (1.5pt);
			\fill  (6,2) circle (1.5pt);
			\fill  (7,2) circle (1.5pt);
			\fill  (7,1) circle (1.5pt);
			
		\end{scriptsize}
	\end{tikzpicture}
	\caption{The stair $S_5$}\label{fig:S5}
\end{figure}

\begin{example}\label{exam: level but initial is not level}
	The stair $S_5$ (see Fig \eqref{fig:S5}) is such that $R/J_\MP$ is level but $R/\ini (J_\MP)$ is not. In fact 
	\[
	\FF(\RR_\MP))=\{\{C_1,C_3,C_5\},\{C_1,C_3,C_6\},\{C_1,C_4,C_6\},\{C_2,C_4,C_6\},\{C_2,C_5\}\}.
	\]
	Hence $\RR_\MP$ is not pure and by Corollary \ref{cor:LevInJ} $R/\ini (J_\MP)$ is not level.
	If $R/J_\MP$ is not level there exists $f\in \soc(R/J_\MP)$ whose degree is strictly less than the rook number of $\MP$. Moreover,  $\ini(f)=\y_F=\y_2\y_5$ (see Lemma \ref{lem:generators-of-inSocle} and Lemma \ref{lem:inSoc}). But $F'=(F\setminus\{C_2\})\cup \{C_1,C_3\}\in \FF(\RR_\MP)$ and by Lemma \ref{lem:FF'}, $\y_F \notin \ini(\soc(R/J_\MP))$. 
\end{example}

\begin{proposition}\label{prop:maximal-ideal-containment}
	Let $\mathcal{P}$ be a path polyomino. If $\reg(S/I_{\MP})=t$, namely $r(\MP)=t$, then $\mathfrak{m}^{t+1}\subseteq J_{\MP}$, where $\mathfrak{m}=(\y_1,\dots,\y_{\ell})\subseteq R$.
\end{proposition}
\begin{proof}
	Since $\mathcal{P}$ is a simple thin polyomino, by \cite[Corollary 3.16]{RR21}, $\reg(S/I_{\MP})=t=r(\mathcal{P})$, where $r(\mathcal{P})$ is the rook number. Let $u=\y_{i_1}\cdots \y_{i_{t+1}}\in \mathfrak{m}^{t+1}$ be any element. Therefore by Corollary \ref{cor:standard-monomial-R/J_P}, $u$ is not a standard monomial in $R/J_{\MP}$. This implies that there are at least two cells $C_{i_j}$, $C_{i_{k}}$ in the attacking positions with $i_j\leq i_k$. 
	If the two cells belong to the first interval (resp. the last interval) $u\in J_0$ (resp. $u\in J_s$).
	Otherwise, there exists a binomial $\y_{i_j}\y_{i_{k}}-\y_{i_j-1}\y_{i_{k}+1}\in J_{\MP}$, to reduce $u$ in $u_1$.

	 Moreover, since $u_1$ has degree $t+1$, as $u$, there are two cells in attacking position induced by the support of $u_1$, too. Then we apply the same procedure again on $u_1$ to get $u_2$. We continue this process until we get $u_l\in J_{0}$ or $u_l\in J_{s}$ so that $u\in J_{\MP}$. Thus, we have $\mathfrak{m}^{t+1}\subseteq J_{\MP}$.
\end{proof}

\begin{corollary}\label{cor:maxdeg}
Let $\MP$ be a polyomino, let $F$ be a facet of $\RR_\MP$ with $|F|=r(\MP)$. Then $\y_{F} \in \soc(R/J_\MP)$.
\end{corollary}

\begin{corollary}\label{cor:level-with-regularity2}
	Let $\MP$ be a path polyomino with $\reg(S/I_{\MP})=2$. Then $\soc(R/J_\MP)$ is generated in degree $2$ and $S/I_{\MP}$ is level.
\end{corollary}
Now we give another class of path polyominoes which are level.
\begin{theorem}\label{thm:interval-length-more-than3}
	Let $\mathcal{P}$ be a path polyomino with $l_1,l_s\geq 2$, and $l_i\geq 3$ for $2\leq i\leq s-1$. Then $S/I_{\MP}$ is level.
\end{theorem}
\begin{proof}
We prove that $S/I_{\MP}$ is level, by proving that  $\soc(R/J_{\MP})$ is generated in degree $r(\MP)$. From Lemma \ref{lem:inSoc} and Corollary \ref{cor:maxdeg} , we have 
\[
\ini (\soc(R/J_\MP)) \subset \soc(R/\ini(J_\MP)),
\]
and for any $F\in \FF(\RR_\MP)$ with $|F|=r(\MP)$, $\y_F\in \ini(\soc(J_\MP))$. Hence it is sufficient to show that for any $F \in \FF(\RR_\MP)$ with $|F|< r(\MP)$ it holds $\y_{F} \notin \ini(\soc(R/J_\MP))$. Since $|F|< r(\MP)$ , then $F$ contains a non-single cell $C_k$. Since $l_k>2$, $C_{k-1}$ and $C_{k+1}$ are single cells, hence 
\[
F'=(F \setminus \{C_{k}\}) \cup \{C_{k-1},C_{k+1}\} \in \FF(\RR_\MP)
\]
and from Lemma \ref{lem:FF'} the claim follows.
\end{proof}

Now we need some technical results to classify the levelness of $S/I_{\MP}$.

\begin{lemma}\label{lem:Ckh}
Let $\MP$ be a path polyomino. Let $F \in \FF(\RR_\MP)$ be such that there exists even $h\geq 2$, $k$, and $\{C_{k},C_{k+2},C_{k+4}\ldots, C_{k+h-2}, C_{k+h}\} \subset F$ (that is $C_k, \ldots, C_{k+h}$ lie on a stair).
\begin{enumerate}
\item Assume that $F'=(F \setminus \{C_{k},C_{k+h}\}) \cup \{C_{k-1},C_{k+h+1}\} \in \RR_{\MP}$. Then for any $j \in \{k,k+1,\ldots, k+h \}$
\[
\y_j\y_{F}=y_j\y_{F'} \mod J_\MP.
\]
\item Assume that $F''=(F \setminus \{C_{k}\}) \cup \{C_{k-1}\}\in \RR_{\MP}$. Then it follows
\[
\y_{k+h+1}\y_{F} = \y_{k+h+2}\y_{F''} \mod J_\PP
 \]
\item Assume that $F'''=(F \setminus \{C_{k+h}\}) \cup \{C_{k+h+1}\} \in \RR_{\MP}$.Then it follows
\[
\y_{k-1}\y_{F} = \y_{k-2}\y_{F'''} \mod J_\PP
\]

\end{enumerate}

\end{lemma}
\begin{proof}
Without loss of generality, we assume that $F=\{C_{k},C_{k+2},\ldots, C_{k+h}\}$.\\
$(1).$ We proceed by induction on $h$. Moreover, to simplify the notation,  all the equalities are equalities modulo $J_\MP$.\\
Let $h=2$, then $F=\{C_{k},C_{k+2}\}$. We prove that $\y_j\y_{k}\y_{k+2}=\y_{j}\y_{k-1}\y_{k+3}$  for $j=k,k+1,k+2$.
\begin{itemize}
\item If $j=k$, then $\y_{k}^2\y_{k+2}= \y_{k-1}\y_{k+1}\y_{k+2} =\y_{k-1}\y_{k}\y_{k+3}$.
\item If $j=k+1$, then $\y_{k}\y_{k+1}\y_{k+2}=\y_{k-1}\y_{k+2}^2=\y_{k-1}\y_{k+1}\y_{k+3}$.
\item If $j=k+2$, then $\y_{k}\y_{k+2}^2= \y_{k}\y_{k+1}\y_{k+3}= \y_{k-1}\y_{k+2}\y_{k+3}$.
\end{itemize}
Hence the assertion follows.
We assume that $h > 2$ and that the thesis holds true for $h-2$, that is the thesis holds true for any facet of $h/2$  cells of the form $\{D_{t},D_{t+2}\ldots, D_{t+h-2} \}$. 
We consider 
\[
F=\{C_{k},C_{k+2},\ldots,C_{k+h-2}, C_{k+h}\}, \ \ F'=\{C_{k-1},C_{k+2},\ldots C_{k+h-2}, C_{k+h+1}\}.
\]
We prove that for any $j \in \{k,\ldots, k+h\}$, $\y_j \y_F=\y_j \y_{F'}$ modulo $J_\MP$. We consider 
\[
F_1=F\setminus \{C_{k+h}\},\text{ and }  F_1'=(F_1 \setminus \{C_{k},C_{k+h-2}\}) \cup \{C_{k-1},C_{k+h-1}\},
\]
and $F_1$, $F_1'$ are faces because $F$ and $F'$ are faces by hypothesis. 
By inductive hypothesis we have that for any $j \in \{k,\ldots, k+h-2\}$, $\y_j \y_{F_1}=\y_{j}\y_{F'_1}$  and hence 
\[
\y_j\y_F=\y_j\y_{F_1}\y_{k+h}=\y_j\y_{F_1'}\y_{k+h}=\y_j\frac{\y_{F_1'}}{\y_{k+h-1}} \y_{k+h-1} \y_{k+h} =
\]
\[
=\y_{j}\y_{k-1}\y_{k+2} \cdots \y_{k+h-4} \y_{k+h-2} \y_{k+h+1} =\y_j\y_{F'}.
\] 

We are left with the cases $j \in \{k+h-1,k+h\}$. If $j=k+h-1$, then $\y_{k+h-1}\y_{F}$ is equal to
\[
\y_{k}\y_{k+2}\cdots \y_{k+h-2} \y_{k+h-1}\y_{k+h} =\y_{k}\y_{k+2}\cdots \y_{k+h-2}^2 \y_{k+h+1} = (\y_{F_1}  \y_{k+h-2})\y_{k+h+1}.
\]
We apply the inductive hypothesis and we get that $\y_{k+h-2} \y_{F_1}=\y_{k+h-2}\y_{F'_1}$,  that is
\[
  (\y_{k+h-2} \y_{F_1} )\y_{k+h+1}=\y_{k+h-2} \y_{k-1}\y_{k+2}\cdots \y_{k+h-4} \y_{k+h-1} \y_{k+h+1} = \y_{k+h-1} \y_{F'}.
\]
If $j=k+h$, then $\y_{k+h}\y_{F}$ is equal to
\[
\y_{k}\y_{k+2}\cdots \y_{k+h-2} \y_{k+h}^2= \y_{k}\y_{k+2}\cdots \y_{k+h-2} \y_{k+h-1} \y_{k+h+1} = (\y_{F_1}  \y_{k+h-1})\y_{k+h+1}.
\]
We apply the inductive hypothesis and we get that $\y_{k+h-1} \y_{F_1}=\y_{k+h-1}\y_{F'_1}$, that is  
\[
 ( \y_{k+h-1} \y_{F_1} )\y_{k+h+1}=\y_{k+h-1} \y_{k-1}\y_{k+2}\cdots \y_{k+h-4} \y_{k+h-1} \y_{k+h+1} =
\]
\[
 \y_{k-1}\y_{k+2}\cdots \y_{k+h-4} \y_{k+h-1}^2 \y_{k+h+1} =\y_{k-1}\y_{k+2} \cdots \y_{k+h-4} \y_{k+h-2} \y_{k+h} \y_{k+h+1}=\y_{k+h} \y_{F'}.
\]
Since $\y_j \y_{F}=\y_{j}\y_{F'}$ modulo $J_\MP$, the assertion follows.

$(2)$. We proceed by induction on $h$. If $h=2$, then $F=\{C_{k},C_{k+2}\}$. We prove that $\y_{k+3} \y_{k}\y_{k+2}=\y_{k-1}\y_{k+2}\y_{k+4} $. We have,
\[
\y_{k}\y_{k+2}\y_{k+3} =\y_{k}\y_{k+1}\y_{k+4} =\y_{k-1}\y_{k+2}\y_{k+4}.
\]
We assume that $h > 2$ and that the thesis holds true for any facet of $h/2$  cells of the form $\{D_{t},D_{t+2}\ldots, D_{t+h-2} \}$. We consider 
\[
F=\{C_{k},C_{k+2},\ldots, C_{k+h}\}, \ \ F''=\{C_{k-1},C_{k+2},\ldots C_{k+h}\}.
\]
We prove that $\y_{k+h+1}\y_F=\y_{k+h+2}\y_{F'}$. We have 
\[
\y_{k}\y_{k+2}\cdots \y_{k+h-2} \y_{k+h}\y_{k+h+1} =\y_{k}\y_{k+2}\cdots \y_{k+h-2} \y_{k+h-1}\y_{k+h+2}
\]
by applying the inductive hypothesis on $\{C_{k},C_{k+2},\ldots, C_{k+h-2}\}$, we get 
\[
\y_{k}\y_{k+2}\cdots \y_{k+h-2} \y_{k+h-1}\y_{k+h+2} =\y_{k-1}\y_{k+2}\cdots \y_{k+h-2} \y_{k+h}\y_{k+h+2} = \y_{k+h+2} \y_{F''},
\]
as desired.\\
$(3)$ Similarly to $(2)$.
\end{proof}

\begin{corollary}\label{cor:3step}
Let $\MP$ be a path polyomino. Let $F \in \FF(\RR_\MP)$ be such that there exists $h\geq 2$, $\ell \geq 2$, $k$,  and $\{C_{k},C_{k+2},\ldots, C_{k+h}, C_{k+h+3},C_{k+h+5},\ldots, C_{k+h+\ell+3}\} \subset F$ with $C_k,\ldots, C_{k+h+\ell+3}$ lying on a stair and such that $F'=(F \setminus \{C_{k},C_{k+h+\ell+3}\}) \cup \{C_{k-1},C_{k+h+\ell+4}\} \in \RR_{\MP}$. Then for $j \in \{k+h+1,k+h+2\}$
\[
\y_j(\y_{F}-\y_{F'})\in J_\MP.
\]
\end{corollary}
\begin{proof}
 We assume that the equalities are modulo $J_\PP$. We prove that for $j \in \{k+h+1,k+h+2\}$, $\y_j\y_F=\y_j\y_{F'}$. Let $j=k+h+1$, we obtain that $\y_j\y_F$ is 
\[
\y_{k}\y_{k+2}\cdots \y_{k+h} (\y_{k+h+1}) \y_{k+h+3}\y_{k+h+5} \cdots \y_{k+h+l+3}.
\]
We apply Lemma \ref{lem:Ckh}.(2) to get 
\[
\y_{k}\y_{k+2}\cdots \y_{k+h} \y_{k+h+1} = \y_{k-1}\y_{k+2}\cdots \y_{k+h} \y_{k+h+2},
\]
and Lemma \ref{lem:Ckh}.(3) to get 
\[
 \y_{k+h+2} \y_{k+h+3}\y_{k+h+5} \cdots \y_{k+h+l+3}=  \y_{k+h+1} \y_{k+h+3}\y_{k+h+5} \cdots \y_{k+h+l+4}.
\]
Hence the thesis follows. The case $j=k+h+2$ similarly follows by applying Lemma \ref{lem:Ckh}.(3) and Lemma \ref{lem:Ckh}.(2).
\end{proof}

\begin{proposition}\label{prop:stairs-not-level}
Let $\MP= S_{\lambda}$ such that $\lambda=4,6$ or $\lambda\geq 8$.
Then $S/I_\MP$ is not level.
\end{proposition}
\begin{proof}
If $\MP$ is the stair $S_\lambda$, then $|\MP|=\lambda+1$.\\
If $\lambda$ is even and $\lambda\geq 4$, then $\lambda=2k$ and $r(\MP)=k+1$. We consider $F=\{C_{2}, C_{4} , \ldots, C_{\lambda} \}\in \FF(\RR_{\MP})$ and has cardinality $k$. Let $F'=\{C_{1},C_{4},C_{6}, \ldots ,C_{\lambda-2}, C_{\lambda+1}\}$. We prove that $\y_{F}-\y_{F'} \in \soc (R/J_\MP)$, namely that for any $j \in \{1,\ldots, \lambda+1\}$ 
\[
\y_{j}(\y_{F}-\y_{F'}) \in J_\MP.
\]
For $j=1,\lambda+1$, it easily follows since $\y_1^2,\y_1\y_2\in J_0\subseteq J_\MP$ and $\y_{\lambda}\y_{\lambda+1},\y_{\lambda+1}^2\in J_{\lambda}\subseteq J_\MP$, while for $j \in \{2,3,\ldots,\lambda\}$ this is the case of Lemma \ref{lem:Ckh}.(1). Hence, the first part of the assertion follows.\\

If $\lambda$ is odd and $\lambda\geq 9$, then $\lambda=2m-1$ and $r(\MP)=m$. We first deal with the case $m$ odd.
We consider 
\[
F=\{C_2,C_4\} \cup \{C_{2k-1}: 4\leq k \leq m \} \mbox{ and } F'=\{C_1,C_5\} \cup  \{C_{2k-1}: 4\leq k  \leq m \} 
\]
Since $m$ is odd, then $m-3=|\{C_{2k-1}: 4\leq k \leq m \}|$ is even and $m-3=2t$ for some $t$. We further consider
\[
F''=\{C_{2k} : k \in \{1,2, \ldots, m-t-1,m-t+1, \ldots,m\}\}
\]
\[
F'''=\{C_1\} \cup \{C_{2k} : k \in \{2,\ldots, m-t-2\}\} \cup\{C_{2(m-t)-1}\} \cup \{C_{2k} : k \in \{m-t+1,\ldots, m\}\}
\]
We claim that for any $j\in \{1,\ldots, \lambda+1\}$ we have 
\[
\y_{j}(\y_{F}-\y_{F'}-\y_{F''}+\y_{F'''}) \in J_\MP.
\]
From now on we deal with the cases $j=2h-1$ for $h=1,\ldots, m$. The case $j=2h$ equivalently follows.

\noindent \textbf{Claim:} Let $G$ be the set 
\[
G=\{C_{2k-1}: 4\leq k \leq m \}.
\]
For $j=2h-1$ with $4 \leq h \leq m$, we claim that 
\begin{equation}\label{eq:Gtilde}
\y_{2h-1}\y_G = \y_{2h-1} \y_{H}  \mod J_\PP,
\end{equation}
where, given 
\[
G_h=\{C_{2k}: 3 \leq k \leq h-1 \}\cup \{C_{2k-1}: h+1\leq k \leq m-h+3 \} \cup \{C_{2k}: m-h+4 \leq k \leq m \},
\]
and
\[
H=\begin{cases}G_h &\mbox{if } h \in \{4,\ldots ,  \frac{m-3}{2}+3\} \\ G_{m-h+4} &\mbox{if } h \in \{ \frac{m-3}{2}+4, \ldots, m\} \end{cases}
\]
\textit{Proof of the claim:} If $4 \leq h \leq m$, then from Lemma \ref{lem:Ckh}.(1) we have 
\[
\y_{2h-1}\y_G = \y_{2h-1}\y_{G'}  \mod J_\PP
\] 
where  $G'=\{ C_6 \} \cup \{C_{2k-1}: 5\leq k \leq m-1 \} \cup \{C_{2m}\}=G_4$. Moreover, if $5\leq h \leq m-1$, then we can do the same argument for $\{C_{2k-1}: 5\leq k \leq m-1 \}$. We can repeat the same argument until we get
\begin{equation}\label{eq:Gtilde}
\y_{2h-1}\y_G = \y_{2h-1} \y_{G_h}  \mod J_\PP,
\end{equation}
where 
\[
G_h=\{C_{2k}: 3 \leq k \leq h-1 \}\cup \{C_{2k-1}: h+1\leq k \leq m-h+3 \} \cup \{C_{2k}: m-h+4 \leq k \leq m \}.
\]

This set makes sense for $h \in \{4,\ldots ,  \frac{m-3}{2}+3=t+3\}$. If $h \in \{t+4, \ldots, m \}$, then $h'=m-h+4 \in \{4,\ldots, t+3 \} $, and 
\[
\y_{2h-1}\y_G = \y_{2h-1} \y_{G_{h'}} \mod J_\PP.
\]
This proves the claim.

We want to compute $\y_{j}\y_{F}$, $\y_{j}\y_{F'}$, $\y_{j}\y_{F''}$, $\y_{j}\y_{F'''}$ for $j=1,\ldots , m$. Given $j \in \{2h-1,2h\}$, for $4 \leq h \leq m$ we use Equation \eqref{eq:Gtilde} to get
\[
\y_{j}\y_{F}=\y_j\y_{2}\y_{4}\y_G = \y_{j} \y_{2}\y_{4} \y_{G_h}
\]
\[
\y_{j}\y_{F'}=\y_j\y_{1}\y_{5}\y_G = \y_{j} \y_{1}\y_{5} \y_{G_h}.
\]

We divide into three cases:
\begin{enumerate}
\item $h < t+3$;
\item $h= t+3$;
\item $h> t+3$.
\end{enumerate}

(1). We have that for any $h= \{2,3 \ldots ,m-t-1=t+2\}$, from Lemma \ref{lem:Ckh}.(1) applied to the set $\{C_{2k}:k\in \{1,2,\dots,m-t-1\}\}$, we have $\y_{j}(\y_{F''}-\y_{F'''})\in J_\MP$.
Hence we only should control that we have $\y_{j}(\y_{F}-\y_{F'})\in J_\MP$. 
Moreover from Lemma \ref{lem:Ckh} and Corollary \ref{cor:3step} for $j=2,\ldots, 6$, $\y_{j}(\y_{F}-\y_{F'})\in J_\MP$.
Since  
\[
h < t+3 \ \  \Rightarrow \ \  2h < m+3  \ \ \Rightarrow \ \ h< m-h+3 \ \ \Rightarrow  \ \ h+1 \leq m-h+3,
\]
then $\{C_{2k-1}: h+1\leq k \leq m-h+3 \}\neq \varnothing$.
Then $C_{2h-2}$ and $C_{2h+1}$ have a $3$ step difference, that is we are in the hypotheses of Corollary \ref{cor:3step}, and from Corollary \ref{cor:3step}, we get 
\[
 \y_{j} \y_{2}\y_{4} \y_{G_h} = \y_j \y_{1} \y_{4} \frac{\y_{G_h}}{\y_{2(m-h+3)-1}} \y_{2(m-h+3)}.
\]
On the other hand, we compute 
\[
\y_{j}\y_{F'}=\y_j \y_{1}\y_{5}\y_G = \y_{j} \y_{1}\y_{5} \y_{G_h}
\]
since $\Big(\prod\limits_{k=3}^{h-1} \y_{2k} \Big)|  \y_{G_h}$, then we apply Lemma \ref{lem:Ckh}.(3) to $\y_5\Big(\prod\limits_{k=3}^{h-1} \y_{2k}\Big) = \y_{4}\Big(\prod\limits_{k=3}^{h-2} \y_{2k}\Big) \y_{2h-1} $ and hence
\[
\y_{j} \y_{1}\y_{5} \y_{G_h} = \y_{j} \y_{1}\y_{4} \frac{\y_{G_h}}{\y_{2h-2}} \y_{2h-1}.
\]
Since $\Big(\prod\limits_{k=h+1}^{m-h+3} \y_{2k-1}\Big) | \y_{G_h}$, then $\Big(\prod\limits_{k=h}^{m-h+3} \y_{2k-1} \Big) | \y_{2h-1}\y_{G_h}$ and since $j \in \{2h-1,2h\}$, we apply Lemma \ref{lem:Ckh}.(1) to get 
\[
\y_{j}\Big(\prod\limits_{k=h}^{m-h+3} \y_{2k-1} \Big) = \y_{j} \y_{2h-2}\Big( \prod\limits_{k=h+1}^{m-h+2} \y_{2k-1} \Big) \y_{2(m-h+3)}
\]
hence
\[
\y_{j} \y_{1}\y_{4} \frac{\y_{G_h}}{\y_{2h-2}} \y_{2h-1} = \y_{j} \y_{1}\y_{4} \frac{\y_{G_h}}{\y_{2h-2} \y_{2(m-h+3)-1}} \y_{2h-2}\y_{2(m-h+3)} =\y_{j} \y_{1}\y_{4} \frac{\y_{G_h}}{ \y_{2(m-h+3)-1}}\y_{2(m-h+3)}
\]
hence $\y_j(\y_{F}- \y_{F'}) = 0$. \\
(2) In the case $h=t+3$, we have $m-h+4=m-t+1=h+1$ and hence $h-1=m-t-1$ and $\{C_{2k-1}: h+1\leq k \leq m-h+3 \}= \varnothing$. Therefore,
\[
G_{m-t}=\{C_{2k}: 3 \leq k \leq m-t-1 \} \cup \{C_{2k}: m-t+1 \leq k \leq m \}
\]
and
\[
\y_j \y_F = \y_j \y_2 \y_4 \y_{G_{m-t}} = \y_j \y_{F''}
\]
Similarly, from Lemma \ref{lem:Ckh}.(3)  we have $\y_5 \prod\limits_{k=3}^{m-t-1} \y_{2k} = \y_4 \Big( \prod\limits_{k=3}^{m-t-2} \y_{2k} \Big) \y_{2(m-t)-1}$ and 
\[
\y_j \y_{F'} = \y_j \y_1 \y_5 \y_{G_{m-t}} = \y_1 \y_4 \frac{\y_{G_{m-t}}}{\y_{2(m-t)-2}} \y_{2(m-t)-1} =\y_j \y_{F'''},
\]
That is $\y_j (\y_F-\y_{F'}-\y_{F''}+\y_{F'''}) = 0$ modulo $J_\MP$.\\
(3) In the case $h> t+3$, we set $h'=m-h+4$, we have
\[
G_{h'}=\{C_{2k}: 3 \leq k \leq h'-1 \}\cup \{C_{2k-1}: h+1\leq k \leq m-h+3 \} \cup \{C_{2k}: m-h'+4 \leq k \leq m \}, 
\]
hence and from Lemma \ref{lem:Ckh}.(3) $\y_j \Big(\prod\limits_{k= m-h'+4 }^{m-1} \y_{2k} \Big) = \y_j \Big(\prod\limits_{k= m-h'+4 }^{m-2} \y_{2k} \Big)\y_{2m-1} $
and hence
\[
\y_j \y_2 \y_4 \y_{G_{h'}} = 0
\]
because $\y_{2m-1}\y_{2m}| \y_j \y_2 \y_4 \y_{G_{h'}} $. The same can be proved for $\y_{F'}$ and since $h\geq m-t+1$, then also for $\y_{F''}$ and $\y_{F'''}.$ This concludes the case $m$ odd.\\
If $m$ is even, then similar arguments hold, with 
\[
F=\{C_2,C_4\} \cup \{C_{2k-1}: 4\leq k \leq m-1 \} \cup \{C_{2m}\}\]
\[ 
F'=\{C_1,C_5\} \cup  \{C_{2k-1}: 4\leq k  \leq m-1 \} \cup \{C_{2m}\},
\]
and given $m-4=|\{C_{2k-1}: 4\leq k \leq m-1 \}|=2t$, we have
\[
F''=\{C_{2k} : k \in \{1,2, \ldots, m-t-1,m-t+1, \ldots,m\}\}
\]
\[
F'''=\{C_1\} \cup \{C_{2k} : k \in \{2,\ldots, m-t-3\}\} \cup\{C_{2(m-t-2)-1}\} \cup \{C_{2k} : k \in \{m-t,\ldots, m\}\}
\]
Also in this case one can verify that
\[
\y_{j}(\y_{F}-\y_{F'}-\y_{F''}+\y_{F'''}) \in J_\MP.
\]

\end{proof}
Let $\MP$ be a stair of length $\lambda$ with $\lambda=4,6$ or $\lambda\geq 8$ and $\MC=\{I_1,\dots,I_{\lambda}\}$ be its set of maximal cell intervals. If $l_1=2=l_{\lambda}$ i.e., $\MP=S_\lambda$, then by Proposition \ref{prop:stairs-not-level}, $S/I_{\MP}$ is not level. Next we consider the case when $l_1>2$ or $l_{\lambda}>2$ i.e., $\MP=\tilde{S}_\lambda$.
\begin{corollary}\label{cor:stair-not-level-with-length>2}
	Let $\MP=\tilde{S}_\lambda$ be a stair of length $\lambda$ with $\lambda=4,6$ or $\lambda\geq 8$ and $\MC=\{I_1,\dots,I_{\lambda}\}$ be its set of maximal cell intervals. Let $l_1>2$ or $l_{\lambda}>2$. Then $S/I_{\MP}$ is not level.
\end{corollary}
\begin{proof}
	Let $I_1=\{C_1,\dots,C_{l_1}\}$, $I_i=\{C_{l_1+i-2},C_{l_1+i-1}\}$ for $2\leq i\leq \lambda-1$ and $I_\lambda=$ \\ $\{C_{l_1+\lambda-2},C_{l_1+\lambda-1},\dots,C_{l_1+\lambda-3+l_{\lambda}}\}.$ Also let $F_i=\{C_{l_1},C_{l_1+2},\dots,C_{l_1+2i}\}$ for $i\geq 1$. 
	
	\noindent \textbf{Claim:} $\y_j\y_{F_i}= \y_{j-1}\y_{F'_{i}}$, where $F_i'=(F_i\setminus \{C_{l_1+2i}\})\cup \{C_{l_1+2i+1}\}$ for $2\leq j\leq l_1$.
	
	\noindent
	\textit{Proof of the claim:} Let $2\leq j\leq l_1$. Then
	\begin{align*}
		\y_j\y_{F_i}=(\y_j\y_{l_1})\y_{l_1+2}\cdots \y_{l_1+2i} & =\y_{j-1}(\y_{l_1+1}\y_{l_1+2})\cdots \y_{l_1+2i} \\
		& =\y_{j-1}\y_{l_1}(\y_{l_1+3}\y_{l_1+4})\cdots \y_{l_1+2i} \\
		& = \cdots \\
		& =\y_{j-1}\y_{l_1}\y_{l_1+2}\cdots (\y_{l_1+2i-1} \y_{l_1+2i}) \\
		& =\y_{j-1}\y_{l_1}\y_{l_1+2}\cdots \y_{l_1+2i-2} \y_{l_1+2i+1}=\y_{j-1}\y_{F_i'}.
	\end{align*}
	Let us consider the stair $S_\lambda$ with cells $\{C_{l_1-1},C_{l_1},\dots,C_{l_1+\lambda-2},C_{l_1+\lambda-1}\}$ of length $\lambda$.
	First we assume that $\lambda$ is even and $\lambda\geq 4$. Consider $F=\{C_{l_1},C_{l_1+2},\dots,C_{l_1+\lambda-2}\}$ and $F'=\{C_{l_1-1},C_{l_1+2},C_{l_1+4},\dots,C_{l_1+\lambda-4},C_{l_1+\lambda-1}\}$. Then by the proof of Proposition \ref{prop:stairs-not-level}, we have that $\y_j(\y_{F}-\y_{F'})\in J_\MP$ for all $l_1+1\leq j\leq l_1+\lambda-3$. For $j=1$, it is clear that $\y_j\y_{F},\y_j\y_{F'}\in J_0 \subseteq J_\MP$. Let $2\leq j\leq l_1$. Then by the Claim, we have $\y_j\y_{F}=\y_{j-1}\y_{F''}$, where $F''=(F\setminus \{C_{l_1+\lambda-2}\})\cup \{C_{l_1+\lambda-1}\}$.
	Also, we have $$\y_j\y_{F'}=(\y_j\y_{l_1-1})\y_{l_1+2}\cdots \y_{l_1+\lambda-4} \y_{l_1+\lambda-1}=\y_{j-1}\y_{F''}.$$
	Therefore, $\y_j(\y_{F}-\y_{F'})\in J_\MP$ for all $1\leq j\leq l_1$. Similarly, one can show that $\y_j(\y_{F}-\y_{F'})\in J_\MP$ for all $l_1+\lambda-2\leq j\leq l_1+\lambda-3+l_\lambda$. Thus we get $\y_j(\y_F-\y_{F'})\in J_\MP$ for all $j$, hence $\y_F-\y_{F'}\in \soc(R/J_{\MP})$.
	
	Now assume that $\lambda$ is odd and $\lambda=2m-1$ such that $m$ is odd. We consider as in the proof of Proposition \ref{prop:stairs-not-level}
	\[
	F=\{C_{l_1},C_{l_1+2}\} \cup \{C_{l_1+2k-3}: 4\leq k \leq m \}, F'=\{C_{l_1-1},C_{l_1+3}\} \cup  \{C_{l_1+2k-3}: 4\leq k  \leq m \},
	\]
	\[
	F''=\{C_{l_1-2+2k} : k \in \{1,2, \ldots, m-t-1,m-t+1, \ldots,m\}\} \text{ and } F'''=\{C_{l_1-1}\} \cup
	\]
	\[
	 \{C_{l_1-2+2k} : k \in \{2,\ldots, m-t-2\}\} \cup\{C_{l_1+2(m-t)-3}\} \cup \{C_{l_1-2+2k} : k \in \{m-t+1,\ldots, m\}\}.
	\]
	Note that $\y_1\y_F,\y_1\y_{F'},\y_1\y_{F''},\y_1\y_{F'''}\in J_\MP$. Let $2\leq j\leq l_1$. Then $\y_j\y_{l_1}\y_{l_1+2}= \y_{j-1}\y_{l_1}\y_{l_1+3}$ and $\y_j\y_{l_1-1}\y_{l_1+3}= \y_{j-1}\y_{l_1}\y_{l_1+3}$ which further implies that $\y_j(\y_F-\y_{F'})\in J_\MP$ for $2\leq j\leq l_1$.
	Let $F''_1=\{C_{l_1-2+2k} : k \in \{1,2, \ldots, m-t-1\}\}\subseteq F''$. Then by \textit{Claim}, for $2\leq j\leq l_1$, we have $\y_j\y_{F''_1}=\y_{j-1}\y_{F''_2}$, where $F''_2=(F''_1\setminus \{C_{l_1+2(m-t)-4}\})\cup \{C_{l_1+2(m-t)-3}\}$. Also, for $F'''_1=\{C_{l_1-1}\} \cup \{C_{l_1-2+2k} : k \in \{2,\ldots, m-t-2\}\} \cup\{C_{l_1+2(m-t)-3}\}\subseteq F'''$, we have $\y_j\y_{F'''_1}=\y_{j-1}\y_{F''_2}$. Therefore, $\y_j(\y_{F''}-\y_{F'''})\in J_\MP$ for $1\leq j\leq l_1$. Similarly, one can show that $\y_j(\y_{F}-\y_{F'}),\y_j(\y_{F''}-\y_{F'''})\in J_\MP$ for all $l_1+\lambda-2\leq j\leq l_1+\lambda-3+l_\lambda$. Hence, 
	\[
	\y_{j}(\y_{F}-\y_{F'}-\y_{F''}+\y_{F'''}) \in J_\MP \quad \forall j.
	\]
	This completes the proof.
\end{proof}
\begin{definition}
	A stair $\MP$ of length $\lambda$ with $\lambda=4,6$ or $\lambda \geq 8$ is called a \emph{bad stair}.
\end{definition}

\begin{theorem}\label{thm:badstair}
	Let $\MP$ be a polyomino containing a bad stair. Then $S/I_\MP$ is not level. 
\end{theorem}
\begin{proof}
	Assume that $\MP$ contains a bad stair $S_\lambda$ or $\tilde{S}_\lambda$. Here, we show that if $\MP$ contains $\tilde{S}_\lambda$, then $S/I_\MP$ is not level. The case when $\MP$ contains $S_\lambda$ is also similar. According to Corollary \ref{cor:stair-not-level-with-length>2}, $\tilde{S}_\lambda$ is not level, hence let $f_\lambda \in R_{\tilde{S}_\lambda}/J_{\tilde{S}_\lambda}$ such that $f_\lambda \in \soc (R_{\tilde{S}_\lambda}/J_{\tilde{S}_\lambda})$ and $\deg(f_\lambda) < r(\tilde{S}_\lambda)$.  The stair $\tilde{S}_\lambda$ is embedded in $\MP$ in some intervals $I_{k+1},I_{k+2},\ldots, I_{k+\lambda}$ with $I_{k+1}=\{C_1^{k+1},C^{k+1}_2,\dots,C^{k+1}_{r_{k+1}}\}$ where $l(I_{k+1})=r_{k+1}$. Let $\MQ$ be the collection of cells having maximal intervals $\MC \setminus \{I_{k+1},\ldots, I_{k+\lambda}\}$, in particular it is the union of two path  polyominoes $\MP_1$ and $\MP_2$. Let $\y_{G_i} \in \soc (R_{\MP_i}/J_{\MP_i})$ with $\deg \y_{G_i}=r(\MP_i)$. We now show that 
	$$\y_{G_1}\y_{G_2} f_\lambda \in \soc (R/J_\MP).$$ 
	Since $f_\lambda \in \soc (R_{\tilde{S}_\lambda}/J_{\tilde{S}_\lambda})$, then it follows from the proof of Corollary \ref{cor:stair-not-level-with-length>2} that for all $C_j\in \tilde{S}_\lambda\setminus \{C_1^{k+1}\}$, we have $\y_jf_\lambda\in J_{\tilde{S}_\lambda}\subseteq J_{\MP}$. Also for similar reason, we can show the \text{claim} of Corollary \ref{cor:stair-not-level-with-length>2} for $\y_j$ where $C_j=C_1^{k+1}$. Therefore, we can conclude that $\y_jf_{\lambda}\in J_{\MP}$ for all $C_j\in \tilde{S}_\lambda$. This implies that $\y_j(\y_{G_1}\y_{G_2} f_\lambda) \in J_{\MP}$ for all $C_j\in \tilde{S}_\lambda$. Let $C_j\in \MP_i$ be any cell. Then $\y_j\y_{G_i}\in J_{\MP_i}\subseteq J_{\MP}$ for all $C_j\in \MP_i$ and $i=1,2$. Therefore, $\y_j(\y_{G_1}\y_{G_2} f_\lambda)\in J_{\MP}$ for all $C_j\in \MP_1\sqcup \MP_2$, and hence $\y_{G_1}\y_{G_2} f_\lambda\in \soc (R/J_\MP).$ Since $\deg(\y_{G_1}\y_{G_2} f_{\lambda})<r(\MP_1)+r(\MP_2)+r(\tilde{S}_\lambda)=r(\MP)$, by Corollary \ref{cor:maxdeg}, $\soc (R/J_\MP)$ has elements of at least two different degrees. Thus, $R/J_\MP$ is not level and hence, $S/I_{\MP}$ is also not level.
\end{proof}

\begin{proposition}\label{prop:stairs-level}
	Let $\MP= S_{\lambda}$ be a stair of length $\lambda$. The followings are equivalent:
	\begin{enumerate}
		\item $S/I_\MP$ is level;
		\item $\lambda=2,3,5,7$;
	\end{enumerate} 
\end{proposition}
\begin{proof}
$(1) \Rightarrow (2)$ follows from Proposition \ref{prop:stairs-not-level}.
Even if $(2) \Rightarrow (1)$ can be showed by direct computation, we want to give a direct proof. In the case $\lambda=2,3$, the rook number is $2$, hence the assertion follows from Corollary \ref{cor:level-with-regularity2}. For the case $\lambda=5,7$, we will make use of Lemma \ref{lem:FF'}. In fact we have to prove that any $F \in 	\RR_{\MP}$ with $|F|=r(\MP)-1$ satisfies the hypothesis of Lemma \ref{lem:FF'}.
\begin{figure}[H]
\centering
\begin{subfigure}{0.5\textwidth}
\centering
\begin{tikzpicture}
\draw (0,0)--(2,0);
\draw (0,1)--(3,1);
\draw (1,2)--(4,2);
\draw (2,3)--(4,3);

\draw (0,0)--(0,1);
\draw (1,0)--(1,2);
\draw (2,0)--(2,3);
\draw (3,1)--(3,3);
\draw (4,2)--(4,3);

\node at (0.5,0.5) {$A$};
\node at (1.5,0.5) {$B$};
\node at (1.5,1.5) {$C$};
\node at (2.5,1.5) {$D$};
\node at (2.5,2.5) {$E$};
\node at (3.5,2.5) {$F$};
\end{tikzpicture}
\caption{The stair $S_5$}
\end{subfigure}%
\begin{subfigure}{0.5\textwidth}
\centering
\begin{tikzpicture}
\draw (0,0)--(2,0);
\draw (0,1)--(3,1);
\draw (1,2)--(4,2);
\draw (2,3)--(5,3);
\draw (3,4)--(5,4);

\draw (0,0)--(0,1);
\draw (1,0)--(1,2);
\draw (2,0)--(2,3);
\draw (3,1)--(3,4);
\draw (4,2)--(4,4);
\draw (5,3)--(5,4);

\node at (0.5,0.5) {$A$};
\node at (1.5,0.5) {$B$};
\node at (1.5,1.5) {$C$};
\node at (2.5,1.5) {$D$};
\node at (2.5,2.5) {$E$};
\node at (3.5,2.5) {$F$};
\node at (3.5,3.5) {$G$};
\node at (4.5,3.5) {$H$};
\end{tikzpicture}
\caption{The stair $S_7$}
\end{subfigure}
\caption{}\label{fig:S5S7}
\end{figure}
We refer to the labellings given in Figure \ref{fig:S5S7}.
If $\lambda=5$, then $r(\MP)=3$ and the unique facet of cardinality $2$ is $\{B,E\}$ and both $B$ and $E$ satisfy the hypothesis of Lemma \ref{lem:FF'}.\\
If $\lambda=7$, then $r(\MP)=4$ and the facets of cardinality $3$ are 
\[
\{A,D,G\},\{B,D,G\},\{B,E,G\},\{B,E,H\}.
\] 
The cells that satisfy the hypothesis of Lemma \ref{lem:FF'} are respectively $D$,$G$,$B$ and $E$.
\end{proof}

\begin{theorem}\label{theo:level}
Let $\MP$ be a path polyomino. The followings are equivalent:
\begin{enumerate}
	\item $S/I_\MP$ is level;
    \item $\MP$ does not contain bad stairs.
\end{enumerate}.
\end{theorem}
\begin{proof} From Theorem \ref{thm:badstair}, we have that $(1) \Rightarrow (2)$.
We now prove $(2)\Rightarrow (1)$.
If $\MP$ does not contain maximal interval of length $2$, then by Theorem \ref{thm:interval-length-more-than3}, it is level. Hence assume it contains intervals of length $2$. Since $\MP$ does not contain bad stair, if $\MP$ contains a stair $S_{\lambda}$ or $\tilde{S}_{\lambda}$ then it must be true that $\lambda\in \{2,3,5,7\}$. It is enough to show that there is no element of degree $<r(\MP)$ in $\soc(R/J_{\MP})$. Let $f\in \soc(R/J_\MP)$ be an element of degree $<r(\MP)$ with $\ini(f)=u$. Then by Lemma \ref{lem:inSoc}, $u\in \ini(\soc(R/J_{\MP}))\subseteq \soc(R/\ini(J_\MP))$ which implies that $u$ can be written as $u=\y_F$ for some $F\in \FF(\RR_\MP)$ with $|F|<r(\MP)$. Therefore, $F$ contains a non-single cell $C_k$ in the intersection of two intervals $I_j$ and $I_{j+1}$. If both intervals have a single cell, $C_{k-1}$ and $C_{k+1}$, that is $I_j,I_{j+1}$ is a stair $\tilde{S}_2$, then we are in the hypotheses of Lemma \ref{lem:FF'}, $(F\setminus \{C_k\})\cup \{C_{k-1},C_{k+1}\}\in \FF(\RR_\MP)$ which is a contradiction by Lemma \ref{lem:FF'}. Moreover, if one between $I_j$ or $I_{j+1}$ has no single cell, say $I_{j+1}$, then $I_{j+1}$ belongs to a stair $S_{\lambda}$ with $\lambda \in \{3,5,7\}$. The case $\lambda=3$ can be eliminated by the following observation: since $\lambda_{j+1}=2$, then $\lambda_{j},\lambda_{j+2}>2$ and in particular, the cell $C_{k-1}$ in $I_j$ is single, and one can take $(F\setminus \{C_{k}\}) \cup \{C_{k-1}\} \in  \FF(\RR({\MP}))$. Hence, we are left with the case $\lambda \in \{5,7\}$. Observe that $F$ contains a facet $F'\in \MF(\RR_{S_\lambda})$ with $C_k\in F'$ and $|F'|<r(S_\lambda)$. Then $F'$ is one of the form given in the proof of Proposition \ref{prop:stairs-level}. Therefore, $F'$ contains a $C_m$ such that $(F'\setminus \{C_m\})\cup \{C_{m-1},C_{m+1}\}\in \FF(\RR_{S_\lambda})$. Hence, $(F\setminus \{C_m\})\cup \{C_{m-1},C_{m+1}\}\in \FF(\RR_\MP)$ which is a contradiction by Lemma \ref{lem:FF'}. This completes the proof.
\end{proof}

%

\section{Levelness and Pseudo-Gorensteinnes of simple thin polyominoes}
In \cite{MRR1}, we developed different routines to generate polyominoes and test their primality. 
After a slight modification of the code provided in that paper, that is possible to download from \cite{RRSw}, we generated all simple thin polyominoes, classifying them, by using Macaulay2 (see \cite{M2}), with respect to the following properties:
\begin{description}
\item[(G)] Gorenstein;
\item[(PG)] Pseudo-Gorenstein (not Gorenstein);
\item[(L)] Level (not Gorenstein);
\item[(N)] None of the above.
\end{description}
In  Figure \ref{fig:rank6}, we display all the non-path simple thin polyominoes of rank 6. We observe that they are all level, but not Gorenstein. 
\begin{table}[H]
	\centering
	\begin{tabular}{|c|c|c|c|c|c|c|c|}
		\hline
		{\bf Rank}               &\textbf{4} &\textbf{5} &\textbf{6} &\textbf{7}  &\textbf{8}  &\textbf{9}   &\textbf{10}\\
		\hline\hline 
		Gorenstein        &0 &3 &0  &10 &0   &47  &0 \\
		\hline
		Level             &4 &7 &26 &65 &230 &684 &2383 \\
		\hline
		pseudo-Gorenstein &0 &1 &0  &5  &0   &36  &0 \\
		\hline
		None of the above &0 &0 &1  &2  &20  &48  &302 \\
		\hline
		
	\end{tabular}
	\caption{The partition of all simple thin polyominoes of rank less than or equal to 10}
	\label{tab:table_pg_l}
\end{table}

In the website \cite{RRSw}, it is possible to download all the simple thin polyominoes, with respect to the previous partition, having rank in the set $\{4,\ldots,10\}$.

\begin{remark}\label{rem:odd}
	We observe that the polyominoes of rank $1$, $2$, and $3$ are paths and are studied in the previous sections. In particular, the single cell is Gorenstein, the domino (the polyomino with $2$ cells) is level, and there are $2$ paths with rank $3$: one is level and the other is Gorenstein. By the Table \ref{tab:table_pg_l} we observe that the pseudo-Gorenstein simple thin polyominoes have odd rank in the interval $\leq 10$. 
\end{remark}

Inspired by Remark \ref{rem:odd} and Corollary  \ref{cor:pseudo-gorenstein} we obtain the following
\begin{theorem}\label{thm:pgrank}
	Let $\MP$ a simple thin pseudo-Gorenstein polyomino. Then
	\begin{enumerate}
		\item $\rk \MP=2r-1$ for $r\geq 1$;
		\item $r(\MP)=r$.
	\end{enumerate}
\end{theorem}
\begin{proof}
 We use induction on $\rk \MP$. The cases $\rk \MP=1,2,3,4$ are in Remark \ref{rem:odd} and Table \ref{tab:table_pg_l}. Suppose by induction hypothesis that for a fixed $k$ and for all $\MP$ such that   $1\leq \rk \MP\leq k$ the conditions (1) and (2) hold. Now focus on the case $\rk \MP=k+1$. 
 We recall that any simple thin polyomino has an interval, namely $I$, that is called either a tail or an endcut (see Definition 3.4 of \cite{RR21} and Figure \ref{fig:oper}). 
\begin{figure}[H]
\centering
\begin{subfigure}[t]{0.45\textwidth}
\centering
\resizebox{0.8\textwidth}{!}{
\begin{tikzpicture}
\draw (2,3)--(2,1);
\draw (0,1)--(0,2);
\draw (1,3)--(1,1);
\draw (3,1)--(3,2);

\draw (1,3)--(2,3);
\draw (0,1)--(3,1);
\draw (0,2)--(3,2);

\draw[dashed] (-1,2.5)--(0,2);
\draw[dashed] (-1,0.5)--(0,1);
\draw[dashed] (-1,2.5) arc(90:270:1);


\node at (1.5,2.5){$C$};
\node at (1.5,1.5){$D$};
\node at (0.5,1.5){$D_1$};
\node at (2.5,1.5){$D_2$};

\end{tikzpicture}}
\caption*{$\PP$ with endcut $[C,D]$}
\end{subfigure}\  \ \vrule{}  \ \ %
\begin{subfigure}[t]{0.45\textwidth}
\centering
\resizebox{0.8\textwidth}{!}{
\begin{tikzpicture}
\draw (2,3)--(2,1);
\draw (0,1)--(0,2);
\draw (1,2)--(1,1);
\draw (3,1)--(3,3);

\draw (2,3)--(3,3);
\draw (0,1)--(3,1);
\draw (0,2)--(3,2);

\draw[dashed] (-1,2.5)--(0,2);
\draw[dashed] (-1,0.5)--(0,1);
\draw[dashed] (-1,2.5) arc(90:270:1);


\node at (2.5,2.5){$C$};
\node at (1.5,1.5){$D_2$};
\node at (0.5,1.5){$D_1$};
\node at (2.5,1.5){$D$};

\end{tikzpicture}}
\caption*{$\PP$ with tail $[C,D]$}
\end{subfigure}\\
\begin{subfigure}[t]{1\textwidth}
\centering
\resizebox{0.32\textwidth}{!}{
\begin{tikzpicture}
\draw (0,0)--(2,0);
\draw (0,1)--(2,1);

\draw[dashed] (-1,1.5)--(0,1);
\draw[dashed] (-1,-0.5)--(0,0);
\draw[dashed] (-1,1.5) arc(90:270:1);

\draw (0,0)--(0,1);
\draw (1,0)--(1,1);
\draw (2,0)--(2,1);
%

\node at (0.5,0.5){$D_1$};
\node at (1.5,0.5){$D_2$};
\end{tikzpicture}}
\caption*{The polyomino $\PP'$ }
\end{subfigure}
\caption{}\label{fig:oper}
\end{figure} 
  Now, suppose $\MP$ is pseudo-Gorenstein. Then the length of $I$ is $2$. Moreover, call $\MP'$ the polyomino $\MP\setminus I$ (resp. the polyomino after the collapsing in $I$) if $I$ is a tail (resp. if $I$ is an endcut). It is easy to observe that $\MP'$ is pseudo-Goenstein, too. 
 
 Now, suppose $k+1$ is even, then $\rk \MP'=k-1$ is even, too. By induction hypothesis $\MP'$ is not pseudo-Gorenstein. Hence $\MP$ is not pseudo-Gorenstein.
 
 Let $k+1$ be odd. Then $\MP'$ is pseudo-Gorenstein, and it has a unique facet $F$ of maximum cardinality. Moreover, its  cardinality is $k/2$. Then, there is a unique facet of maximum cardinality $k/2+1$ of $\MP$, that is $F$ with the single cell of $I$.  
 \end{proof}

Motivated by the observation that all the path polyominoes that satisfy Theorem \ref{thm:interval-length-more-than3} have at least a single cell in any interval, and by computational evidence (e.g. all of the polyominoes in Figure \ref{fig:rank6} but (1) and (6)), the following conjecture naturally arises.
\begin{conjecture}\label{conj:generalization}
	Let $\mathcal{P}$ be a simple thin polyomino such that any maximal interval has a single cell. Then $S/I_{\MP}$ is level.
\end{conjecture}
\begin{figure}[H]
\centering
\begin{subfigure}{0.14 \textwidth}
\centering
\resizebox{!}{0.6\textwidth}{
\begin{tikzpicture}
\filldraw[red] (0,0)--(1,0)--(1,0)--(1,2)--(0,2);
\filldraw[red] (0,1)--(2,1)--(2,1)--(2,2)--(0,2);
\filldraw[red] (1,1)--(2,1)--(2,1)--(2,4)--(1,4);
\filldraw[red] (1,2)--(3,2)--(3,2)--(3,3)--(1,3);
\draw (0,0) -- (0,4);
\draw (0,0) -- (4,0);
\draw (1,0) -- (1,4);
\draw (0,1) -- (4,1);
\draw (2,0) -- (2,4);
\draw (0,2) -- (4,2);
\draw (3,0) -- (3,4);
\draw (0,3) -- (4,3);
\draw (4,0) -- (4,4);
\draw (0,4) -- (4,4);

\end{tikzpicture}} 
\caption*{(1)}
\end{subfigure}%
\begin{subfigure}{0.14 \textwidth}
\centering
\resizebox{!}{0.6\textwidth}{
\begin{tikzpicture}
\filldraw[red] (0,0)--(1,0)--(1,0)--(1,3)--(0,3);
\filldraw[red] (0,2)--(3,2)--(3,2)--(3,3)--(0,3);
\filldraw[red] (1,2)--(2,2)--(2,2)--(2,4)--(1,4);
\draw (0,0) -- (0,4);
\draw (0,0) -- (4,0);
\draw (1,0) -- (1,4);
\draw (0,1) -- (4,1);
\draw (2,0) -- (2,4);
\draw (0,2) -- (4,2);
\draw (3,0) -- (3,4);
\draw (0,3) -- (4,3);
\draw (4,0) -- (4,4);
\draw (0,4) -- (4,4);

\end{tikzpicture}}\caption*{(2)}
\end{subfigure}%
\begin{subfigure}{0.14 \textwidth}
\centering
\resizebox{!}{0.6\textwidth}{
\begin{tikzpicture}
\filldraw[red] (0,1)--(2,1)--(2,1)--(2,2)--(0,2);
\filldraw[red] (1,0)--(2,0)--(2,0)--(2,4)--(1,4);
\filldraw[red] (1,2)--(3,2)--(3,2)--(3,3)--(1,3);
\draw (0,0) -- (0,4);
\draw (0,0) -- (4,0);
\draw (1,0) -- (1,4);
\draw (0,1) -- (4,1);
\draw (2,0) -- (2,4);
\draw (0,2) -- (4,2);
\draw (3,0) -- (3,4);
\draw (0,3) -- (4,3);
\draw (4,0) -- (4,4);
\draw (0,4) -- (4,4);

\end{tikzpicture}}\caption*{(3)}
\end{subfigure}%
\begin{subfigure}{0.14 \textwidth}
\centering
\resizebox{!}{0.6\textwidth}{
\begin{tikzpicture}
\filldraw[red] (0,0)--(1,0)--(1,0)--(1,2)--(0,2);
\filldraw[red] (0,1)--(3,1)--(3,1)--(3,2)--(0,2);
\filldraw[red] (1,1)--(2,1)--(2,1)--(2,4)--(1,4);
\draw (0,0) -- (0,4);
\draw (0,0) -- (4,0);
\draw (1,0) -- (1,4);
\draw (0,1) -- (4,1);
\draw (2,0) -- (2,4);
\draw (0,2) -- (4,2);
\draw (3,0) -- (3,4);
\draw (0,3) -- (4,3);
\draw (4,0) -- (4,4);
\draw (0,4) -- (4,4);

\end{tikzpicture}}\caption*{(4)}
\end{subfigure}%
\centering
\begin{subfigure}{0.14 \textwidth}
\centering
\resizebox{!}{0.6\textwidth}{
\begin{tikzpicture}
\filldraw[red] (0,0)--(1,0)--(1,0)--(1,2)--(0,2);
\filldraw[red] (0,1)--(4,1)--(4,1)--(4,2)--(0,2);
\filldraw[red] (1,1)--(2,1)--(2,1)--(2,3)--(1,3);
\draw (0,0) -- (0,4);
\draw (0,0) -- (4,0);
\draw (1,0) -- (1,4);
\draw (0,1) -- (4,1);
\draw (2,0) -- (2,4);
\draw (0,2) -- (4,2);
\draw (3,0) -- (3,4);
\draw (0,3) -- (4,3);
\draw (4,0) -- (4,4);
\draw (0,4) -- (4,4);

\end{tikzpicture}} \caption*{(5)}
\end{subfigure}%
\begin{subfigure}{0.14 \textwidth}
\centering
\resizebox{!}{0.6\textwidth}{
\begin{tikzpicture}
\filldraw[red] (0,0)--(1,0)--(1,0)--(1,2)--(0,2);
\filldraw[red] (0,1)--(3,1)--(3,1)--(3,2)--(0,2);
\filldraw[red] (1,1)--(2,1)--(2,1)--(2,3)--(1,3);
\filldraw[red] (2,0)--(3,0)--(3,0)--(3,2)--(2,2);
\draw (0,0) -- (0,3);
\draw (0,0) -- (3,0);
\draw (1,0) -- (1,3);
\draw (0,1) -- (3,1);
\draw (2,0) -- (2,3);
\draw (0,2) -- (3,2);
\draw (3,0) -- (3,3);
\draw (0,3) -- (3,3);

\end{tikzpicture}}\caption*{(6)}
\end{subfigure}%
\begin{subfigure}{0.14 \textwidth}
\centering
\resizebox{!}{0.6\textwidth}{
\begin{tikzpicture}
\filldraw[red] (0,0)--(1,0)--(1,0)--(1,5)--(0,5);
\filldraw[red] (0,2)--(2,2)--(2,2)--(2,3)--(0,3);
\draw (0,0) -- (0,5);
\draw (0,0) -- (5,0);
\draw (1,0) -- (1,5);
\draw (0,1) -- (5,1);
\draw (2,0) -- (2,5);
\draw (0,2) -- (5,2);
\draw (3,0) -- (3,5);
\draw (0,3) -- (5,3);
\draw (4,0) -- (4,5);
\draw (0,4) -- (5,4);
\draw (5,0) -- (5,5);
\draw (0,5) -- (5,5);

\end{tikzpicture}}\caption*{(7)}
\end{subfigure}
\end{figure}
\begin{figure}[H]
\centering
\begin{subfigure}{0.14 \textwidth}
\centering
\resizebox{!}{0.6\textwidth}{
\begin{tikzpicture}
\filldraw[red] (0,1)--(3,1)--(3,1)--(3,2)--(0,2);
\filldraw[red] (1,0)--(2,0)--(2,0)--(2,4)--(1,4);
\draw (0,0) -- (0,4);
\draw (0,0) -- (4,0);
\draw (1,0) -- (1,4);
\draw (0,1) -- (4,1);
\draw (2,0) -- (2,4);
\draw (0,2) -- (4,2);
\draw (3,0) -- (3,4);
\draw (0,3) -- (4,3);
\draw (4,0) -- (4,4);
\draw (0,4) -- (4,4);

\end{tikzpicture}}\caption*{(8)}
\end{subfigure}%
\begin{subfigure}{0.14 \textwidth}
\centering
\resizebox{!}{0.6\textwidth}{
\begin{tikzpicture}
\filldraw[red] (0,0)--(1,0)--(1,0)--(1,2)--(0,2);
\filldraw[red] (0,1)--(4,1)--(4,1)--(4,2)--(0,2);
\filldraw[red] (2,1)--(3,1)--(3,1)--(3,3)--(2,3);
\draw (0,0) -- (0,4);
\draw (0,0) -- (4,0);
\draw (1,0) -- (1,4);
\draw (0,1) -- (4,1);
\draw (2,0) -- (2,4);
\draw (0,2) -- (4,2);
\draw (3,0) -- (3,4);
\draw (0,3) -- (4,3);
\draw (4,0) -- (4,4);
\draw (0,4) -- (4,4);

\end{tikzpicture}} \caption*{(9)}
\end{subfigure}%
\begin{subfigure}{0.14 \textwidth}
\centering
\resizebox{!}{0.6\textwidth}{
\begin{tikzpicture}
\filldraw[red] (0,0)--(1,0)--(1,0)--(1,4)--(0,4);
\filldraw[red] (0,0)--(2,0)--(2,0)--(2,1)--(0,1);
\filldraw[red] (0,2)--(2,2)--(2,2)--(2,3)--(0,3);
\draw (0,0) -- (0,4);
\draw (0,0) -- (4,0);
\draw (1,0) -- (1,4);
\draw (0,1) -- (4,1);
\draw (2,0) -- (2,4);
\draw (0,2) -- (4,2);
\draw (3,0) -- (3,4);
\draw (0,3) -- (4,3);
\draw (4,0) -- (4,4);
\draw (0,4) -- (4,4);

\end{tikzpicture}} \caption*{(10)}
\end{subfigure}%
\begin{subfigure}{0.14 \textwidth}
\centering
\resizebox{!}{0.6\textwidth}{
\begin{tikzpicture}
\filldraw[red] (0,0)--(1,0)--(1,0)--(1,5)--(0,5);
\filldraw[red] (0,1)--(2,1)--(2,1)--(2,2)--(0,2);
\draw (0,0) -- (0,5);
\draw (0,0) -- (5,0);
\draw (1,0) -- (1,5);
\draw (0,1) -- (5,1);
\draw (2,0) -- (2,5);
\draw (0,2) -- (5,2);
\draw (3,0) -- (3,5);
\draw (0,3) -- (5,3);
\draw (4,0) -- (4,5);
\draw (0,4) -- (5,4);
\draw (5,0) -- (5,5);
\draw (0,5) -- (5,5);

\end{tikzpicture}} \caption*{(11)}
\end{subfigure}%
\begin{subfigure}{0.14 \textwidth}
\centering
\resizebox{!}{0.6\textwidth}{
\begin{tikzpicture}
\filldraw[red] (0,0)--(1,0)--(1,0)--(1,4)--(0,4);
\filldraw[red] (0,1)--(3,1)--(3,1)--(3,2)--(0,2);
\draw (0,0) -- (0,4);
\draw (0,0) -- (4,0);
\draw (1,0) -- (1,4);
\draw (0,1) -- (4,1);
\draw (2,0) -- (2,4);
\draw (0,2) -- (4,2);
\draw (3,0) -- (3,4);
\draw (0,3) -- (4,3);
\draw (4,0) -- (4,4);
\draw (0,4) -- (4,4);

\end{tikzpicture}} \caption*{(12)}
\end{subfigure}%
\begin{subfigure}{0.14 \textwidth}
\centering
\resizebox{!}{0.6\textwidth}{
\begin{tikzpicture}
\filldraw[red] (0,0)--(1,0)--(1,0)--(1,3)--(0,3);
\filldraw[red] (0,1)--(4,1)--(4,1)--(4,2)--(0,2);
\draw (0,0) -- (0,4);
\draw (0,0) -- (4,0);
\draw (1,0) -- (1,4);
\draw (0,1) -- (4,1);
\draw (2,0) -- (2,4);
\draw (0,2) -- (4,2);
\draw (3,0) -- (3,4);
\draw (0,3) -- (4,3);
\draw (4,0) -- (4,4);
\draw (0,4) -- (4,4);

\end{tikzpicture}} \caption*{(13)}
\end{subfigure}%
\begin{subfigure}{0.14 \textwidth}
\centering
\resizebox{!}{0.6 \textwidth}{
\begin{tikzpicture}
\filldraw[red] (0,0)--(1,0)--(1,0)--(1,3)--(0,3);
\filldraw[red] (0,1)--(3,1)--(3,1)--(3,2)--(0,2);
\filldraw[red] (2,0)--(3,0)--(3,0)--(3,2)--(2,2);
\draw (0,0) -- (0,3);
\draw (0,0) -- (3,0);
\draw (1,0) -- (1,3);
\draw (0,1) -- (3,1);
\draw (2,0) -- (2,3);
\draw (0,2) -- (3,2);
\draw (3,0) -- (3,3);
\draw (0,3) -- (3,3);

\end{tikzpicture}}\caption*{(14)} 
\end{subfigure}%
\caption{Simple Thin Polyominoes of Rank 6}\label{fig:rank6}
\end{figure}
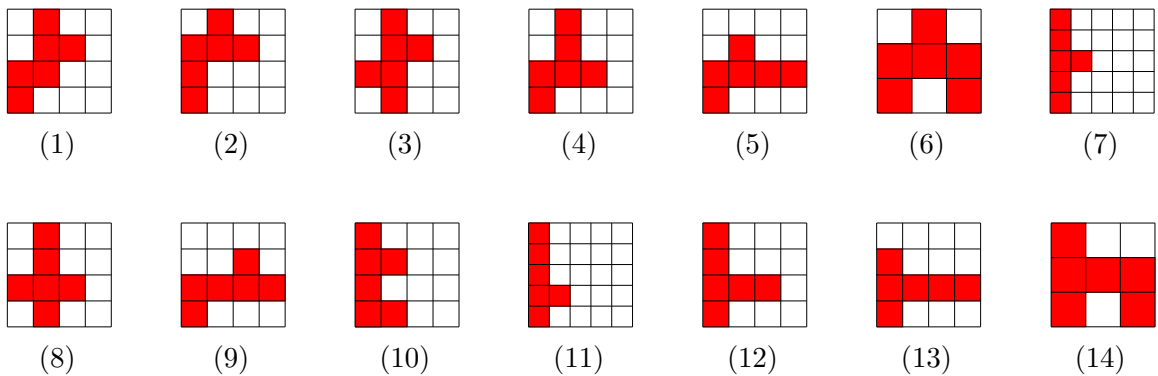

In conclusion, it is of interest the following 
\begin{question}
Is it possible to generalize the concept of (bad) stair to characterize level or pseudo-Gorenstein simple thin polyominoes?
\end{question}

\bibliographystyle{plain}

\end{document}